\documentclass[12pt]{amsart}
\usepackage{hyperref,cleveref,amssymb}

\theoremstyle{plain}
\newtheorem{theorem}{Theorem}[section]
\newtheorem{proposition}[theorem]{Proposition}
\newtheorem{corollary}[theorem]{Corollary}
\newtheorem{lemma}[theorem]{Lemma}

\theoremstyle{definition}
\newtheorem{remark}[theorem]{Remark}
\newtheorem{question}[theorem]{Question}

\newtheorem{example}[theorem]{Example}

\newcommand{\abs}[1]{\lvert#1\rvert}
\newcommand{\norm}[1]{\lVert#1\rVert}
\newcommand{\bigabs}[1]{\bigl\lvert#1\bigr\rvert}
\newcommand{\bignorm}[1]{\bigl\lVert#1\bigr\rVert}

\newcommand{\Bignorm}[1]{\Bigl\lVert#1\Bigr\rVert}
\newcommand{\term}[1]{{\textit{\textbf{#1}}}}   % To introduce a term
\renewcommand{\mid}{\::\:}

\newcommand{\goeso}{\xrightarrow{\mathrm{o}}}	% o-convergence
\newcommand{\goesun}{\xrightarrow{\mathrm{un}}} % un-convergence
\newcommand{\goesuo}{\xrightarrow{\mathrm{uo}}}	% uo-convergence
\newcommand{\goesnorm}{\xrightarrow{\norm{\cdot}}}	% norm convergence
\newcommand{\goesmu}{\xrightarrow{\mu}}	%convergence in measure
\newcommand{\goesw}{\xrightarrow{\mathrm{w}}}	% weak convergence
\newcommand{\goesws}{\xrightarrow{\mathrm{w}^*}}	% weak* convergence
\newcommand{\goesae}{\xrightarrow{\mathrm{a.e.}}}	% a.e. convergence

\DeclareSymbolFont{bbold}{U}{bbold}{m}{n}
\DeclareSymbolFontAlphabet{\mathbbold}{bbold}
\def\one{\mathbbold{1}}

\DeclareMathOperator{\Range}{Range}
\DeclareMathOperator{\Span}{span}
\DeclareMathOperator{\supp}{supp}

\renewcommand{\le}{\leqslant}
\renewcommand{\ge}{\geqslant}

\begin{document}

\title[Unbounded norm topology]
{Unbounded Norm Topology\\ in Banach Lattices}

\author{M. Kandi\'c}
\address{Faculty of Mathematics and Physics, University of Ljubljana,
  Jadranska 19, 1000 Ljubljana, Slovenia}
\email{marko.kandic@fmf.uni-lj.si}

\author{M.A.A. Marabeh} \address{Department of Mathematics, Middle East Technical University, 06800 Ankara, Turkey.}
\email{mohammad.marabeh@metu.edu.tr, m.maraabeh@gmail.com}

\author{V.G. Troitsky}
\address{Department of Mathematical and Statistical Sciences,
         University of Alberta, Edmonton, AB, T6G\,2G1, Canada.}
\email{troitsky@ualberta.ca}

\thanks{The first author was supported in part by grant P1-0222 of
  Slovenian Research Agency. The second author was supported by Middle
  East Technical University grant number BAP-01-01-2016-001.  The
  third author was supported by an NSERC grant.}

% 46B42 Banach lattices
% 46A40 Ordered topological linear spaces, vector lattices

\date{\today}

\begin{abstract}
  A net $(x_\alpha)$ in a Banach lattice $X$ is said to un-converge to
  a vector $x$ if $\bignorm{\abs{x_\alpha-x}\wedge u}\to 0$ for every
  $u\in X_+$. In this paper, we investigate un-topology, i.e., the
  topology that corresponds to un-convergence. We show that
  un-topology agrees with the norm topology iff $X$ has a strong unit.
  Un-topology is  metrizable iff $X$ has a quasi-interior
  point. Suppose that $X$ is order continuous, then un-topology is locally
  convex iff $X$ is atomic. An order continuous Banach lattice $X$ is
  a KB-space iff its closed unit ball $B_X$ is un-complete. For a Banach
  lattice $X$, $B_X$ is un-compact iff $X$ is an atomic KB-space. We
  also study un-compact operators and the relationship between
  un-convergence and weak*-convergence.
\end{abstract}

\maketitle

\section{Introduction and preliminaries}

For a net $(x_\alpha)$ in a vector lattice $X$, we write
$x_\alpha\goeso x$ if $(x_\alpha)$ \term{converges} to $x$ \term{in
  order}. That is, there is a net $(u_\gamma)$, possibly over a
different index set, such that $u_\gamma\downarrow 0$ and for every
$\gamma$ there exists $\alpha_0$ such that
$\abs{x_\alpha-x}\le u_\gamma$ whenever $\alpha\ge\alpha_0$. We write
$x_\alpha\goesuo x$ and say that $(x_\alpha)$ \term{uo-converges} to
$x$ if $\abs{x_\alpha-x}\wedge u\goeso 0$ for every $u\in X_+$; ``uo''
stands for ``unbounded order''. For a net $(x_\alpha)$ in a normed
lattice $X$, we write $x_\alpha\goesnorm x$ if $(x_\alpha)$ converges
to $x$ in norm. We write $x_\alpha\goesun x$ and say that $(x_\alpha)$
\term{un-converges} to $x$ if $\abs{x_\alpha-x}\wedge u\goesnorm 0$
for every $u\in X_+$; ``un'' stands for ``unbounded norm''.

A variant of uo-convergence was originally introduced in
\cite{Nakano:48}, while the term ``uo-convergence'' was first coined
in \cite{DeMarr:64}. Relationships between uo, weak, and weak*
convergences were investigated in
\cite{Wickstead:77,GaoX:14,Gao:14}. Relationships between
uo-convergence and almost everywhere convergence were investigated and
applied in \cite{GaoX:14,Emelyanov:16,GTX}. We refer the
reader to \cite{GTX} for a further review of properties of
uo-convergence. Un-convergence was introduced in \cite{Troitsky:04}
and further investigated in \cite{DOT}. For unexplained terminology on
vector and Banach lattices we refer the reader to
\cite{Abramovich:02,Aliprantis:06}. All vector lattices are assumed to
be Archimedean.

Let us start by briefly going over some of the known properties of
these modes of convergence; we refer the reader to \cite{GTX,DOT} for
details. Both uo-convergence and un-convergence respect linear and
lattice operations; limits are unique. In particular,
$x_\alpha\goesuo x$ iff $\abs{x_\alpha-x}\goesuo 0$; similarly,
$x_\alpha\goesun x$ iff $\abs{x_\alpha-x}\goesun 0$. For order bounded
nets, uo-convergence agrees with order convergence while
un-convergence agrees with norm convergence. It follows that order
intervals are uo- and un-closed. For sequences in $L_p(\mu)$, where
$1\le p<\infty$ and $\mu$ is a finite measure, it is easy to see that
uo-convergence agrees with convergence almost everywhere, see, e.g.,
\cite[Example~2]{DeMarr:64}. Under the same assumptions,
un-convergence agrees with convergence in measure, see
\cite[Example~23]{Troitsky:04}. We write $L_p$ for $L_p[0,1]$.

Suppose that $X$ is a vector lattice. By \cite[Corollary~3.6]{GTX},
every disjoint sequence in $X$ is uo-null. Recall that a sublattice
$Y$ of $X$ is \term{regular} if the inclusion map preserves suprema
and infima of arbitrary subsets.  It was shown
in~\cite[Theorem~3.2]{GTX} that uo-convergence is stable under passing
to and from regular sublattices. That is, if $(y_\alpha)$ is a net in
a regular sublattice $Y$ of $X$ then $y_\alpha\goesuo 0$ in $Y$ iff
$y_\alpha\goesuo 0$ in $X$ (in fact, this property characterizes
regular sublattices). 

It is clear that if $X$ is an order continuous normed lattice then
uo-convergence implies un-convergence. Let $X$ be a Banach lattice and
$(x_n)$ a un-null sequence in $X$. Then $(x_n)$ has a uo-null
subsequence by Proposition~4.1 of~\cite{DOT}. A disjoint sequence need
not be un-null. For example, the standard unit sequence $(e_n)$ in
$\ell_\infty$ is not un-null. However, a un-null sequence has an
asymptotically disjoint subsequence. More precisely, we have the
following.

\begin{theorem}(\cite[Theorem~3.2]{DOT})\label{KP}
  Let $(x_\alpha)$ be a un-null net. There is an increasing sequence
  of indices $(\alpha_k)$ and a disjoint sequence $(d_k)$ such that
  $x_{\alpha_k}-d_k\goesnorm 0$.
\end{theorem}

While uo-convergence need not be given by a topology, it was observed
in \cite{DOT} that un-convergence is topological. For every
$\varepsilon>0$ and non-zero $u\in X_+$, put
\begin{displaymath}
  V_{\varepsilon,u}=\bigl\{x\in X\mid \bignorm{\abs{x}\wedge u}<\varepsilon\bigr\}.
\end{displaymath}
The collection of all sets of this form is a base of zero
neighborhoods for a topology, and the convergence in this topology
agrees with un-convergence. We will refer to this topology
as \emph{un-topology}.

Every time a new linear topology is discovered, one is expected to ask
several natural questions: is this topology metrizable? Is it
locally-convex?  Complete? Can one characterize (relatively) compact
sets? Is this topology stronger or weaker than other known topologies?
In this paper, we study these and similar questions for
un-topology. In other words, our motivation for this paper is to
investigate topological properties of un-topology.

Throughout this paper, $X$ will be assumed to be a Banach
lattice, unless specified otherwise. We write $B_X$ for the closed
unit ball of $X$.  It was observed in~\cite{DOT} that
$x_\alpha\goesun x$ implies $\norm{x}\le\liminf\norm{x_\alpha}$. This
yields that $B_X$ is un-closed.

The following facts will be used throughout the paper.

\begin{lemma}\label{monot}
  \begin{enumerate}
  \item\label{monot-uo} If $(x_\alpha)$ is an increasing net in a
    vector lattice $X$ and $x_\alpha\goesuo x$ then
    $x_\alpha\uparrow x$;
  \item\label{monot-un} If $(x_\alpha)$ is an increasing net in a
    normed lattice $X$ and $x_\alpha\goesun x$ then
    $x_\alpha\uparrow x$ and $x_\alpha\goesnorm x$.
  \end{enumerate}
\end{lemma}

\begin{proof}
  Without loss of generality, $x_\alpha\ge 0$ for all $\alpha$;
  otherwise, pick any index $\alpha_0$ and consider the net
  $(x_\alpha-x_{\alpha_0})_{\alpha\ge\alpha_0}$, which converges to
  $x-x_{\alpha_0}$. Since lattice operations are uo- and un-continuous,
  we have $x\ge 0$.

  \eqref{monot-uo} Take any $z\in X_+$. It follows from uo-continuity
  of lattice operations that $x_\alpha\wedge z\goesuo x\wedge z$.
  Since the net $(x_\alpha\wedge z)$ is order bounded and increasing,
  this yields $x_\alpha\wedge z\goeso x\wedge z$ and, therefore
  $x_\alpha\wedge z\uparrow x\wedge z$. It follows that
  $x_\alpha\wedge z\le x$ for every $\alpha$ and every $z\in
  X_+$. Applying this with $z=x_\alpha$ we get
  $x_\alpha\le x$. Thus, the net $(x_\alpha)$ is order bounded and,
  therefore, $x_\alpha\goeso x$, hence $x_\alpha\uparrow x$.

  \eqref{monot-un} The proof is similar and uses the fact that every
  monotone norm convergent net converges in order to the same
  limit. We note that $x_\alpha\wedge z\goesnorm x\wedge z$ and,
  therefore,   $x_\alpha\wedge z\uparrow x\wedge z$ for every $z\in
  X_+$. It follows that the net $(x_\alpha)$ is order bounded, which
  yields $x_\alpha\goesnorm x$ and, therefore, $x_\alpha\uparrow x$.
\end{proof}

Recall that \cite[Question~2.14]{DOT} asks whether $x_\alpha\goesun 0$
implies that there exists an increasing sequence of indices
$(\alpha_k)$ such that $x_{\alpha_k}\goesun 0$. The following
counterexample was kindly provided to us by E.~Emelyanov.

\begin{example}\label{emelyanov}
  Let $\Omega$ be an uncountable set; let $X$ be the closed sublattice
  of $\ell_\infty(\Omega)$ consisting of all the functions with
  countable support. For $\omega\in\Omega$, we write $e_\omega$ for
  the characteristic function of $\{\omega\}$. 

  Let $\Lambda$ be the set of all countable subsets of $\Omega$,
  ordered by inclusion. For each $\alpha\in\Lambda$, pick any
  $\omega\notin\alpha$ and put $x_\alpha=e_\omega$. We claim that
  $x_\alpha\goesun 0$. Indeed, let $u\in X_+$; let $\alpha_0$ be the
  support of $u$. Then $x_\alpha\wedge u=0$ whenever
  $\alpha\ge\alpha_0$.

  On the other hand, let $(\omega_k)$ be any sequence in $\Omega$; we
  claim that the sequence $(e_{\omega_k})$ is not un-null. Indeed, put
  $\beta=\{\omega_k\mid k\in\mathbb N\}$ and let $u$ be the
  characteristic function of $\beta$. Then
  $e_{\omega_k}\wedge u=e_{\omega_k}$ for every $k$; hence it does not
  converge in norm to zero.

  In particular, if $(\alpha_k)$ is an increasing sequence of indices
  in $\Lambda$ then $(x_{\alpha_k})$ is not un-null.
\end{example}

Let $e\in X_+$. Recall that the band $B_e$ generated by $e$ is norm
closed and contains the principal ideal $I_e$; hence
$I_e\subseteq\overline{I_e}\subseteq B_e$. Recall also that
\begin{itemize}
\item $e$ is a \term{strong unit} when $I_e=X$; equivalently, for
  every $x\ge 0$ there exists $n\in\mathbb N$ such that $x\le ne$;
\item $e$ is a \term{quasi-interior point} if $\overline{I_e}=X$;
  equivalently, $x\wedge ne\goesnorm x$ for every $x\in X_+$;
\item $e$ is a \term{weak unit} if $B_e=X$;
  equivalently, $x\wedge ne\uparrow x$ for every $x\in X_+$.
\end{itemize}
In particular,
strong unit $\Rightarrow$ quasi-interior point 
$\Rightarrow$ weak unit.

\section{Strong units}

It is easy to see that each $V_{\varepsilon,u}$ is solid. It is also
absorbing, that is, for every $x\in X$ there exists $\lambda>0$ such
that $\lambda x\in V_{\varepsilon,u}$. The following lemma is a
dichotomy: it says that $V_{\varepsilon,u}$ is either ``very small''
or ``very large''.

\begin{lemma}\label{dichotomy}
  Let $\varepsilon>0$, and $0\ne u\in X_+$. Then $V_{\varepsilon,u}$
  is either contained in $[-u,u]$ or contains a non-trivial ideal.
\end{lemma}

\begin{proof}
  Suppose that $V_{\varepsilon,u}$ is not contained in $[-u,u]$. Then
  there exists $x\in V_{\varepsilon,u}$ such that
  $x\notin[-u,u]$. Replacing $x$ with $\abs{x}$, we may assume that
  $x>0$. Let $y=(x-u)^+$; then $y>0$. It is an easy exercise to show
  that
  \begin{math}
    (\lambda y)\wedge u\le x\wedge u
  \end{math}
  for every $\lambda\ge 0$; it follows that $\lambda y\in
  V_{\varepsilon,u}$. Since $V_{\varepsilon,u}$ is solid, it contains
  the principal ideal $I_y$.
\end{proof}

\begin{lemma}\label{V-bdd-su}
  If $V_{\varepsilon,u}$ is contained in $[-u,u]$ then $u$ is a strong unit.
\end{lemma}

\begin{proof}
  Let $x\in X_+$. There exists $\lambda>0$ such that $\lambda x\in
  V_{\varepsilon,u}$, hence $\lambda x\in [-u,u]$. It follows that $u$
  is a strong unit.
\end{proof}

Recall that if $e$ is a positive vector in $X$ then the principal
ideal $I_e$ equipped with the norm
\begin{displaymath}
  \norm{x}_e=\inf\bigl\{\lambda>0\mid\abs{x}\le\lambda e\bigr\}
\end{displaymath}
is lattice isometric to $C(K)$ for some compact Hausdorff space $K$,
with $e$ corresponding to the constant one function $\one$; see, e.g.,
Theorems~3.4 and~3.6 in~\cite{Abramovich:02}. If $e$ is a strong unit
in $X$ then $I_e=X$; it is easy to see that in this case
$\norm{\cdot}_e$ is equivalent to the original norm; it follows that $X$ is
lattice and norm isomorphic to $C(K)$.

It is easy to see that if $x_\alpha\goesnorm x$ then
$x_\alpha\goesun x$, so norm topology generally is stronger than
un-topology. 

\begin{theorem}\label{su}
  Let $X$ be a Banach lattice. The following are equivalent.
  \begin{enumerate}
  \item Un-topology agrees with norm topology;
  \item $X$ has a strong unit.
  \end{enumerate}
\end{theorem}

\begin{proof}
  Suppose that un-topology and norm topology agree. It follows that
  $V_{\varepsilon,u}$ is contained in $B_X$ for some $\varepsilon>0$
  and $u>0$. By \Cref{dichotomy}, we conclude that $V_{\varepsilon,u}$
  is contained in $[-u,u]$; hence $u$ is a strong unit by \Cref{V-bdd-su}.

  Suppose now that $X$ has a strong unit. Then $X$ is lattice and norm
  isomorphic to $C(K)$ for some compact Hausdorff space $K$.  Without
  loss of generality, $X=C(K)$. It follows from $x_\alpha\goesun 0$
  that $\abs{x_\alpha}\wedge\one\goesnorm 0$. Since the norm in $C(K)$
  is the $\sup$-norm, it is easy to see that $x_\alpha\goesnorm 0$.
\end{proof}

\section{Quasi-Interior points and metrizability}

Given a net $(x_\alpha)$ in a vector lattice with a weak unit $e$,
then $x_\alpha\goesuo x$ iff $\abs{x_\alpha-x}\wedge e\goeso 0$; see,
e.g., \cite[Corollary~3.5]{GTX} (this was proved in~\cite{Kaplan:97}
in the special case when the lattice is order complete). That is, it
suffices to test uo-convergence on a weak unit. Lemma~2.11
in~\cite{DOT} provides a similar statement for un-convergence and
quasi-interior points. We now prove that this property actually
characterizes quasi-interior points.

\begin{theorem}\label{qip}
  Let $e\in X_+$. The following are equivalent.
  \begin{enumerate}
  \item\label{qip-un} $e$ is a quasi-interior point;
  \item\label{qip-net} For every net $(x_\alpha)$ in $X_+$, if $x_\alpha\wedge
    e\goesnorm 0$ then $x_\alpha\goesun 0$;
  \item\label{qip-seq} For every sequence $(x_n)$ in $X_+$, if $x_n\wedge
    e\goesnorm 0$ then $x_n\goesun 0$.
  \end{enumerate}
\end{theorem}

\begin{proof}
  The implication \eqref{qip-un}$\Rightarrow$\eqref{qip-net} was
  proved
  in~\cite[Lemma~2.11]{DOT}.
  \eqref{qip-net}$\Rightarrow$\eqref{qip-seq} is trivial. This leaves
  \eqref{qip-seq}$\Rightarrow$\eqref{qip-un}.

  Suppose \eqref{qip-seq}. Fix $x\in X_+$. We need to show that
  $x\wedge ne\goesnorm x$ or, equivalently $(x-ne)^+\goesnorm 0$ as a
  sequence of $n$. Put $u=x\vee e$. The ideal $I_u$ is lattice
  isomorphic (as a vector lattice) to $C(K)$ for some compact space
  $K$, with $u$ corresponding to $\one$. Since $x,e\in I_u$, we may
  consider $x$ and $e$ as elements of $C(K)$. Note that $x\vee e=\one$
  implies that $x$ and $e$ never vanish simultaneously. 

  For each $n\in\mathbb N$, we define
  \begin{displaymath}
    F_n=\bigl\{t\in K\mid x(t)\ge ne(t)\bigl\}
    \text{ and }
    O_n=\bigl\{t\in K\mid x(t)>ne(t)\bigl\}.
  \end{displaymath}
  Clearly, $O_n\subseteq F_n$, $O_n$ is open, and $F_n$ is closed.

  \emph{Claim 1}: $F_{n+1}\subseteq O_n$. Indeed, let $t\in
  F_{n+1}$. Then $x(t)\ge (n+1)e(t)$. If $e(t)>0$ then $x(t)>ne(t)$,
  so that $t\in O_n$. If $e(t)=0$ then $x(t)>0$, hence $t\in O_n$.

  By Urysohn's Lemma, we find $z_n\in C(K)$ such that $0\le z_n\le x$,
  $z_n$ agrees with $x$ on $F_{n+1}$ and vanishes outside of $O_n$. We can
  also view $ z_n$ as an element of $X$.

  \emph{Claim 2}: $n(z_n\wedge e)\le x$. Let $t\in K$. If $t\in O_n$
  then $n(z_n\wedge e)(t)\le ne(t)<x(t)$. If $t\notin O_n$ then
  $z_n(t)=0$, so that the inequality is satisfied trivially.

  \emph{Claim 3}: $\bigl(x-(n+1)e\bigr)^+\le z_n$. Again, let
  $t\in K$. If $t\in F_{n+1}$ then
  $\bigl(x-(n+1)e\bigr)^+\le x(t)=z_n(t)$. If $t\notin F_{n+1}$ then
  $x(t)<(n+1)e(t)$, so that $\bigl(x-(n+1)e\bigr)^+(t)=0$ and the
  inequality is satisfied trivially.

  Now, Claim 2 yields $0\le z_n\wedge e\le\frac{1}{n}x\goesnorm 0$,
  so that $z_n\wedge e\goesnorm 0$. By assumption, this yields
  $z_n\goesun 0$. Since $0\le z_n\le x$ for every $n$, the sequence
  $(z_n)$ is order bounded and, therefore, $z_n\goesnorm 0$. Now
  Claim~3 yields $\bigl(x-(n+1)e\bigr)^+\goesnorm 0$, which concludes
  the proof.
\end{proof}

\begin{theorem}\label{metriz}
  Un-topology is metrizable iff $X$ has a quasi-interior point. If $e$
  is a quasi-interior point then $d(x,y)=\bignorm{\abs{x-y}\wedge e}$
  is a metric for un-topology.
\end{theorem}

\begin{proof}
  Suppose that $e\in X_+$ is a quasi-interior point and put
  $d(x,y)=\bignorm{\abs{x-y}\wedge e}$ for $x,y\in X$. It can be
  easily verified that this defines a metric on $X$. Indeed,
  $d(x,x)=0$ and $d(x,y)=d(y,x)$ for every $x,y\in X$. If $d(x,y)=0$
  then $\abs{x-y}\wedge e=0$, hence $\abs{x-y}=0$ because $e$ is a
  weak unit, so that $x=y$. The triangle inequality follows from the
  fact that
  \begin{displaymath}
    \abs{x-z}\wedge e\le\abs{x-y}\wedge e+\abs{y-z}\wedge e.
  \end{displaymath}
  Note also that
  $x_\alpha\goesun x$ iff $d(x_\alpha,x)\to 0$ for every net
  $(x_\alpha)$ in $X$.

  Conversely, suppose that un-topology is metrizable; let $d$ be a
  metric for it. For each $n$, let $B_{\frac1n}$ be the ball of radius
  $\frac1n$ centred at zero for the metric, that is,
  \begin{displaymath}
    B_{\frac1n}=\bigl\{x\in X\mid d(x,0)\le \tfrac1n\bigr\}.
  \end{displaymath}
  Since $B_{\frac1n}$ is a neighborhood of zero for the un-topology,
  it contains $V_{\varepsilon_n,u_n}$ for some $\varepsilon_n>0$ and
  $u_n>0$. Let $M_n=2^n\norm{u_n}+1$; then the series
  $e=\sum_{n=1}^\infty\frac{u_n}{M_n}$ converges. Note that $M_n>1$
  and $u_n\le M_ne$ for every $n$. We claim that $e$ is a
  quasi-interior point.

  It suffices that \Cref{qip}\eqref{qip-net} is satisfied.
  Suppose that $x_\alpha\wedge e\goesnorm 0$ for some net $(x_\alpha)$
  in $X_+$. Fix $n$. It follows from
  \begin{displaymath}
    x_\alpha\wedge u_n\le (M_nx_\alpha)\wedge(M_ne)=M_n(x_\alpha\wedge
    e)\goesnorm 0
  \end{displaymath}
  that $x_\alpha\wedge u_n\goesnorm 0$. Then
  there exists $\alpha_0$ such that $\norm{x_\alpha\wedge
    u_n}<\varepsilon_n$ whenever $\alpha\ge\alpha_0$. Consequently,
  $x_\alpha$ is in $V_{\varepsilon_n,u_n}$ and, therefore, in
  $B_{\frac1n}$. It follows that $x_\alpha\to 0$ in the metric, hence
  $x_\alpha\goesun 0$. 
\end{proof}

Note that a linear Hausdorff topological space is metrizable iff it is
first countable, i.e., has a countable base of neighborhoods of zero,
see, e.g., \cite[pp.~49]{Kelley:63}. Therefore, Theorem~\ref{metriz}
implies, in particular, that un-topology is first countable iff
$X$ has a quasi-interior point. This should be compared with
Corollary~2.13 and Question~2.14 in~\cite{DOT} (we now know from
Example~\ref{emelyanov} that Question~2.14 has a negative answer).

\begin{proposition}\label{un-str-metr}
  Un-topology is stronger than or equal to a metric topology iff $X$
  has a weak unit.
\end{proposition}

\begin{proof}
  Suppose that un-topology is stronger than or equal to a topology
  given by a metric. Construct $e$ as in the second part of the proof
  of \Cref{metriz}. We claim that $e$ is a weak unit. Suppose that
  $x\wedge e=0$. It follows that $x\wedge u_n=0$ for every $n$ and,
  therefore, $x\in V_{\varepsilon_n,u_n}$, hence $x\in
  B_{\frac1n}$. It follows that $x=0$.

  Conversely, let $e\in X_+$ be a weak unit. For $x,y\in X$, define
  $d(x,y)=\bignorm{\abs{x-y}\wedge e}$. As in the first part of the
  proof of \Cref{metriz}, this is a metric and $x_\alpha\goesun x$
  implies $d(x_\alpha,x)\to 0$.
\end{proof}

\subsection*{When is every un-null sequence norm bounded?}
If $X$ has a strong unit then, by \Cref{su}, un-topology agrees with
norm topology, hence every un-null sequence is norm null and, in
particular, norm bounded. This justifies the following question:
\emph{If every un-null sequence in $X$ is norm bounded (or even norm
  null), does this imply that $X$ has a strong unit?} The following
example shows that, in general, the answer in negative.

\begin{example}
  Let $X$ be as in \Cref{emelyanov}. Clearly, $X$ does not have a
  strong unit; it does not even have a weak unit. Yet, every un-null
  sequence in $X$ is norm null. Indeed, suppose that $x_n\goesun
  0$. Let $u$ be the characteristic function of
  $\bigcup_{n=1}^\infty\supp x_n$. By assumption, $\abs{x_n}\wedge
  u\goesnorm 0$. It follows that for every $\varepsilon\in(0,1)$ there
  exists $n_0$ such that for every $n\ge n_0$ we have
  $\bignorm{\abs{x_n}\wedge u}<\varepsilon$. It follows that
  $\norm{x_n}<\varepsilon$. 
\end{example}

However, we will see that the answer is affirmative under certain
additional assumptions.

Recall that every disjoint sequence is uo-null. Thus, if
$\dim X=\infty$, one can take any non-zero disjoint sequence, scale it
to make it norm unbounded, and thus produce a uo-null sequence which
is not norm bounded. However, this trick does not work for un-topology
because a disjoint sequence need not be un-null. Moreover, we have the
following.

\begin{proposition}\label{disj}
  The following are equivalent.
  \begin{enumerate}
  \item\label{disj-oc} $X$ is order continuous;
  \item\label{disj-seq} Every disjoint sequence in $X$ is un-null;
  \item\label{disj-net} Every disjoint net in $X$ is un-null. 
  \end{enumerate}
\end{proposition}

\begin{proof}
  \eqref{disj-oc}$\Rightarrow$\eqref{disj-seq} because every disjoint
  sequence is uo-null and, therefore, un-null. To show that
  \eqref{disj-seq}$\Rightarrow$\eqref{disj-oc}, note that every order
  bounded disjoint sequence is norm null and apply
  \cite[Theorem~4.14]{Aliprantis:06}.

  \eqref{disj-net}$\Rightarrow$\eqref{disj-seq} is trivial. To show
  that \eqref{disj-seq}$\Rightarrow$\eqref{disj-net}, suppose that
  there exists a disjoint net $(x_\alpha)$ which is not un-null. Then
  there exist $\varepsilon>0$ and $u\in X_+$ such that for every
  $\alpha$ there exists $\beta>\alpha$ with
  $\bignorm{\abs{x_\beta}\wedge u}>\varepsilon$. Inductively, we find
  an increasing sequence $(\alpha_k)$ of indices such that
  $\bignorm{\abs{x_{\alpha_k}}\wedge u}>\varepsilon$. Hence, the
  sequence $(x_{\alpha_k})$ is disjoint but not un-null.
\end{proof}

\begin{corollary}
  If $X$ is order continuous and every un-null sequence in $X$ is norm
  bounded then $\dim X<\infty$ (and, therefore, $X$ has a strong
  unit).
\end{corollary}

\begin{proof}
  Suppose $\dim X=\infty$. Then there exists a non-zero disjoint
  sequence in $X$. Scaling it if necessary, we may assume that
  it is not norm bounded. Yet it is un-null. A contradiction.
\end{proof}

Note that Example~2.7 in~\cite{DOT} is an example of a disjoint but
non un-null sequence in an infinite-dimensional Banach lattice which
is not order continuous and lacks a strong unit.

\begin{proposition}
  If $X$ has a quasi-interior point and every un-null sequence is norm
  bounded then $X$ has a strong unit.
\end{proposition}

\begin{proof}
  By \Cref{metriz}, the un-topology on $X$ is metrizable. Fix such a
  metric. As before, for each $n$, let $B_{\frac1n}$ be the ball of
  radius $\frac1n$ centred at zero for the metric. For each $n$,
  $B_{\frac1n}$ contains $V_{\varepsilon_n,u_n}$ for some
  $\varepsilon_n>0$ and $u_n>0$. If
  $V_{\varepsilon_n,u_n}\subseteq[-u_n,u_n]$ for some $n$ then $u_n$
  is a strong unit by \Cref{V-bdd-su}. Otherwise, by \Cref{dichotomy},
  each $V_{\varepsilon_n,u_n}$ contains a non-trivial ideal. Pick any
  $x_n$ in this ideal with $\norm{x_n}=n$. Then the sequence $(x_n)$
  is norm unbounded; yet $x_n\in B_{\frac1n}$ for every $n$, so that
  $x_n\goesun 0$; a contradiction.
\end{proof}

\section{Un-convergence in a sublattice}

Recall that if $(y_\alpha)$ is a net in a regular sublattice $Y$ of a
vector lattice $X$ then $y_\alpha\goesuo 0$ in $Y$ iff
$y_\alpha\goesuo 0$ in $X$. The situation is very different for
un-convergence. Let $Y$ be a sublattice of a normed lattice
$X$ and $(y_\alpha)$ a net in $Y$. If $y_\alpha\goesun 0$ in $X$
then, clearly, $y_\alpha\goesun 0$ in $Y$. However, the following
examples show that the converse fails even for closed ideals or bands.

\begin{example}\label{ex:c0}
  The sequence of the standard unit vectors $(e_n)$ is un-null in
  $c_0$ but not in $\ell_\infty$, even though $c_0$ is a closed ideal in
  $\ell_\infty$. 
\end{example}

\begin{example}
  Let $X=C[-1,1]$ and $Y$ be the set of all $f\in X$ which vanish on
  $[-1,0]$. It is easy to see that $Y$ is a band (though it is not a
  projection band). Let $(f_n)$ be a sequence in $Y_+$ such
  that $\norm{f_n}=1$ and $\supp
  f_n\subseteq[\frac{1}{n+1},\frac1n]$.
  Since $X$ has a strong unit, the un-topology on $X$ agrees with the
  norm topology, hence $(f_n)$ is not un-null in $X$. However,
  it is easy to see that $(f_n)$ is un-null in $Y$.
\end{example}

Nevertheless, there are some good news. Recall that a sublattice $Y$
of a vector lattice $X$ is \term{majorizing} if for every $x\in X_+$
there exists $y\in Y_+$ with $x\le y$.

\begin{theorem}\label{sublat}
  Let $Y$ be a sublattice of a normed lattice $X$ and $(y_\alpha)$ a
  net in $Y$ such that $y_\alpha\goesun 0$ in $Y$. Each of the
  following conditions implies that $y_\alpha\goesun 0$ in $X$.
  \begin{enumerate}
  \item\label{sublat-maj} $Y$ is majorizing in $X$;
  \item\label{sublat-dense} $Y$ is norm dense in $X$;
  \item\label{sublat-projb} $Y$ is a projection band in $X$.
  \end{enumerate}
\end{theorem}

\begin{proof}
  Without loss of generality, $y_\alpha\ge 0$ for every $\alpha$.
  \eqref{sublat-maj} is straightforward. To
  prove~\eqref{sublat-dense}, take $u\in X_+$ and fix
  $\varepsilon>0$. Find $v\in Y_+$ with $\norm{u-v}<\varepsilon$. By
  assumption, $y_\alpha\wedge v\goesnorm 0$. We can find $\alpha_0$
  such that $\norm{y_\alpha\wedge v}<\varepsilon$
  whenever $\alpha\ge\alpha_0$. It follows from $u\le v+\abs{u-v}$
  that
  \begin{math}
    y_\alpha\wedge u\le y_\alpha\wedge v+\abs{u-v},
  \end{math}
  so that
  \begin{displaymath}
    \norm{y_\alpha\wedge u}\le\norm{y_\alpha\wedge v}+\norm{u-v}
    <2\varepsilon.
  \end{displaymath}
  It follows that $y_\alpha\wedge u\goesnorm 0$. Hence
  $y_\alpha\goesun 0$ in $X$.

  To prove~\eqref{sublat-projb}, let $u\in X_+$. Then $u=v+w$ for some
  positive $v\in Y$ and $w\in Y^d$. It follows from $y_\alpha\perp w$ that
  \begin{math}
    y_\alpha\wedge u=y_\alpha\wedge v\goesnorm 0.
  \end{math}
\end{proof}

Recall that every (Archimedean) vector lattice $X$ is majorizing in
its \term{order (or Dedekind) completion} $X^\delta$; see , e.g., \cite[p.~101]{Aliprantis:06}.

\begin{corollary}
  If $X$ is a normed lattice and $x_\alpha\goesun x$ in $X$ then
  $x_\alpha\goesun x$ in the order completion $X^\delta$ of $X$.
\end{corollary}

\begin{corollary}
  If $X$ is a KB-space and $x_\alpha\goesun 0$ in $X$ then
  $x_\alpha\goesun 0$ in $X^{**}$.
\end{corollary}

\begin{proof}
  By \cite[Theorem~4.60]{Aliprantis:06}, $X$ is a projection band in
  $X^{**}$. The conclusion now follows from
  \Cref{sublat}\eqref{sublat-projb}.
\end{proof}

\Cref{ex:c0} shows that the assumption that $X$ is a KB-space cannot
be removed.

\begin{corollary}
  Let $Y$ be a sublattice of an order continuous Banach lattice
  $X$. If $y_\alpha\goesun 0$ in $Y$ then  $y_\alpha\goesun 0$ in $X$.
\end{corollary}

\begin{proof}
  Suppose that $y_\alpha\goesun 0$ in $Y$. By
  \Cref{sublat}\eqref{sublat-maj}, $y_\alpha\goesun 0$ in the ideal
  $I(Y)$ generated by $Y$ in $X$. By
  \Cref{sublat}\eqref{sublat-dense}, $y_\alpha\goesun 0$ in the
  closure $\overline{I(Y)}$ of the ideal. Since $X$ is order
  continuous, $\overline{I(Y)}$ is a projection band in $X$.
  It now follows from \Cref{sublat}\eqref{sublat-projb} that
  $y_\alpha\goesun 0$ in $X$.
\end{proof}

\begin{question}
  Let $B$ be a band in $X$. Suppose that every net in $B$ which is
  un-null in $B$ is also un-null in $X$. Does this imply that $B$ is a
  projection band?
\end{question}

\bigskip

\begin{proposition}\label{band-closed}
  Every band in a normed lattice is un-closed.
\end{proposition}

\begin{proof}
  Let $B$ be a band and $(x_\alpha)$ a net in $B$ such that
  $x_\alpha\goesun x$. Fix $z\in B^d$. Then $\abs{x_\alpha}\wedge z=0$
  for every $\alpha$. Since lattice operations are un-continuous, we
  have $\abs{x}\wedge z=0$. It follows that $x\in B^{dd}=B$.
\end{proof}

\begin{remark}\label{pb}
  Let $B$ be a projection band a normed lattice $X$. We write
  $P_B$ for the corresponding band projection. It follows easily from
  $0\le P_B\le I$ that if $x_\alpha\goesun x$ in $X$ then
  $P_Bx_\alpha\goesun P_Bx$ both in $X$ and in $B$.
\end{remark}

\subsection*{Dense band decompositions.}
Let $X$ be a Banach lattice. By a \term{dense band
  decomposition} of $X$ we mean a family $\mathcal B$ of pairwise
disjoint projection bands in $X$ such that the linear span of all of
the bands in $\mathcal B$ is norm dense in $X$. 

\begin{lemma}
  Let $\mathcal B$ be a family of pairwise disjoint projection bands
  in a Banach lattice $X$. $\mathcal B$ is a dense band decomposition
  of $X$ iff for every $x\in X$ and every $\varepsilon>0$ there exist
  $B_1,\dots,B_n$ in $\mathcal B$ such that
  $\bignorm{x-\sum_{i=1}^nP_{B_i}x}<\varepsilon$.
\end{lemma}

\begin{proof}
  Suppose that $\mathcal B$ is a dense band decomposition
  of $X$. Let $x\in X$ and $\varepsilon>0$. By assumption, we can find
  distinct bands $B_1,\dots,B_n$ and vectors $x_1\in B_1,\dots,x_n\in
  B_n$ such that
  \begin{math}
    \bignorm{x-\sum_{i=1}^nx_i}<\varepsilon.
  \end{math}
  Put $Q=I-\sum_{i=1}^nP_{B_i}$. Then $Q$ is also a band projection,
  hence it is a lattice homomorphism and $0\le Q\le I$. Note also that
  $Qx_i=0$ for $i=1,\dots,n$. We have
  \begin{displaymath}
    \bigabs{x-\sum_{i=1}^nx_i}
    \ge Q\bigabs{x-\sum_{i=1}^nx_i}
    =\bigabs{Qx-\sum_{i=1}^nQx_i}
    =\bigabs{x-\sum_{i=1}^nP_{B_i}x}.
  \end{displaymath}
  It follows that $\bignorm{x-\sum_{i=1}^nP_{B_i}x}<\varepsilon$.

  The converse implication is trivial.
\end{proof}

Our definition of a disjoint band decomposition is partially motivated
by following fact.

\begin{theorem}(\cite[Proposition~1.a.9]{Lindenstrauss:79})\label{wunit-decomp}
  Every order continuous Banach lattice admits a dense band
  decomposition $\mathcal B$ such that each band in $\mathcal B$ has a
  weak unit.
\end{theorem}

It is easy to see that if $X$ is an order continuous Banach lattice
and $\mathcal B$ is a pairwise disjoint collection of bands such that
$x=\sup\{P_{B}x\mid B\in\mathcal B\}$ for every $x\in X_+$ then
$\mathcal B$ is a dense band decomposition.

\begin{theorem}\label{band-decomp}
  Suppose that $\mathcal B$ is a dense band decomposition of a Banach
  lattice $X$. Then $x_\alpha\goesun x$ in $X$ iff $P_Bx_\alpha\goesun
  P_Bx$ in $B$ for each $B\in\mathcal B$. 
\end{theorem}

\begin{proof}
  Without loss of generality, $x=0$ and $x_\alpha\ge 0$ for every
  $\alpha$.  The forward implication follows immediately from
  \Cref{pb}. To prove the converse, suppose that
  $P_Bx_\alpha\goesun 0$ in $B$ for each $B\in\mathcal B$. Let
  $u\in X_+$; it suffices to show that $x_\alpha\wedge u\goesnorm 0$.
  Fix $\varepsilon>0$. Find $B_1,\dots,B_n\in\mathcal B$
  such that
  \begin{math}
    \bignorm{u-\sum_{i=1}^nP_{B_i}u}<\varepsilon.
  \end{math}
  Since $P_{B_i}x_\alpha\goesun 0$ in $B_i$ as $i=1,\dots,n$, we can
  find $\alpha_0$ such that
  \begin{math}
    \bignorm{P_{B_i}x_\alpha\wedge P_{B_i}u}<\frac{\varepsilon}{n}
  \end{math}
  for every $\alpha\ge\alpha_0$ and every $i=1,\dots,n$.
  It follows from $x_\alpha\wedge P_{B_i}u\in B_i$ that
  $x_\alpha\wedge P_{B_i}u=P_{B_i}x_\alpha\wedge P_{B_i}u$.
  Therefore,
  \begin{multline*}
      \norm{x_\alpha\wedge u}\le
      \Bignorm{x_\alpha\wedge\sum_{i=1}^nP_{B_i}u}+
      \Bignorm{u-\sum_{i=1}^nP_{B_i}u}
      \le\Bignorm{\sum_{i=1}^nx_\alpha\wedge P_{B_i}u}+\varepsilon\\
      =\Bignorm{\sum_{i=1}^nP_{B_i}x_\alpha\wedge P_{B_i}u}+\varepsilon
      \le n\cdot\frac{\varepsilon}{n}+\varepsilon
      \le 2\varepsilon.
  \end{multline*}
\end{proof}

\begin{remark}\label{atoms}
  Recall that a positive non-zero vector $a$ in a vector lattice $X$
  is an \term{atom} if the principal ideal $I_a$ generated by $a$
  coincides with $\Span a$. In this case, $I_a$ is a projection band,
  and the corresponding band projection $P_a$ has form $f_a\otimes a$
  for some positive functional $f_a$, that is, $P_ax=f_a(x)a$. We say
  that $X$ is \term{non-atomic} if it has no atoms. We say that $X$ is
  \term{atomic} if $X$ is the band generated by all the atoms. In the
  latter case,
\begin{math}
  x=\sup\{f_a(x)a\mid a\text{ is an atom}\}
\end{math}
for every $x\in X_+$.
See, e.g., \cite[p.~143]{Schaefer:74}.
\end{remark}

It follows that if $X$ is an order continuous atomic Banach lattice, the
family $\{I_a\mid a\text{ is an atom}\}$ is a dense band decomposition
of $X$. Applying \Cref{band-decomp}, we conclude that in such spaces
un-convergence is exactly the ``coordinate-wise'' convergence:

\begin{corollary}\label{un-atomic}
  Let $X$ be an atomic order continuous Banach lattice. Then
  $x_\alpha\goesun x$ iff $f_a(x_\alpha)\to f_a(x)$ for every atom $a$.
\end{corollary}

\begin{remark}
  The order continuity assumption cannot be removed. Indeed,
  $\ell_\infty$ is atomic, the sequence $(e_n)$ converges to zero
  coordinate-wise, yet it is not un-null.
\end{remark}

The following results extends \cite[Proposition 6.2]{DOT}.

\begin{proposition}\label{w-un}
  The following are equivalent:
  \begin{enumerate}
  \item\label{w-un-net} $x_\alpha\goesw 0$ implies $x_\alpha\goesun 0$ for every
    net $(x_\alpha)$ in $X$;
  \item\label{w-un-seq} $x_n\goesw 0$ implies $x_n\goesun 0$ for every
    sequence $(x_n)$ in $X$;
  \item\label{w-un-aoc} $X$ is atomic and order continuous.
  \end{enumerate}
\end{proposition}

\begin{proof}
  \eqref{w-un-net}$\Rightarrow$\eqref{w-un-seq} is
  trivial. The implication
  \eqref{w-un-seq}$\Rightarrow$\eqref{w-un-aoc} is a part of
  \cite[Proposition 6.2]{DOT}. The implication
  \eqref{w-un-aoc}$\Rightarrow$\eqref{w-un-net} follows from \Cref{un-atomic}.
\end{proof}

\section{AL-representations and local convexity}

In this section, we will show that un-topology on an order continuous
Banach lattice $X$ is locally convex iff $X$ is atomic. Our main tool
is the relationship between un-convergence in $X$ and in an
AL-representation of $X$.

It was observed in \cite[Example~23]{Troitsky:04} that for a net
$(x_\alpha)$ in $L_p(\mu)$ where $\mu$ is a finite measure and
$1\le p<\infty$, one has $x_\alpha\goesun 0$ iff $x_\alpha\goesmu 0$
(i.e., the net converges to zero in measure). Note that this does not
extend to $\sigma$-finite measures.  Indeed, let
$X=L_p(\mathbb R)$ and let $x_n$ be the characteristic function of
$[n,n+1]$. Then $x_n\goesun 0$ but $(x_n)$ does not converge to zero
in measure. On the other hand, let $(x_\alpha)$ be a net in $L_p(\mu)$
where $\mu$ is a $\sigma$-finite measure, let $(\Omega_n)$ be a
countable partition of $\Omega$ into sets of finite measure; it
follows from \Cref{band-decomp} that $x_\alpha\goesun 0$ iff the
restriction of $x_\alpha$ to $\Omega_n$ converges to zero in measure
for every $n$.

Suppose that $X$ is an order continuous Banach lattice with a weak
unit $e$. By \cite[Theorem~1.b.14]{Lindenstrauss:79}, $X$ can be
represented as an ideal of $L_1(\mu)$ for some probability measure
$\mu$. More precisely, there is a lattice isomorphism from $X$ onto a
norm-dense ideal of $L_1(\mu)$; with a slight abuse of notation we
will view $X$ itself as an ideal of $L_1(\mu)$. Moreover, this
representation may be chosen so that $e$ corresponds to $\one$,
$L_\infty(\mu)$ is a norm-dense ideal in $X$, and both inclusions in
$L_\infty(\mu)\subseteq X\subseteq L_1(\mu)$ are continuous. We call
$L_1(\mu)$ an \term{AL-representation} for $X$ and $e$. Let $(x_n)$ be
a sequence in $X$. It was shown in \cite[Remark~4.6]{GTX} that
$x_n\goesuo 0$ in $X$ iff $x_n\goesae 0$ in $L_1(\mu)$. It was shown
in \cite[Theorem~4.6]{DOT} that $x_n\goesun 0$ in $X$ iff
$x_n\goesmu 0$ in $L_1(\mu)$. Since un-topology and the topology of
convergence in measure are both metrizable on $X$ because $X$ has a
weak unit, it follows that these two topologies coincide on $X$. In
particular, $x_\alpha\goesun 0$ in $X$ iff $x_\alpha\goesmu 0$ in
$L_1(\mu)$ for every net $(x_\alpha)$ in $X$. This may also be deduced
from Amemiya's Theorem (see, e.g., Theorem~2.4.8 in
\cite{Meyer-Nieberg:91}) as follows:
\begin{displaymath}
  x_\alpha\goesun 0\text{ in }X\quad\Leftrightarrow\quad
  \norm{x_\alpha\wedge e}_X\to 0
  \quad\overset{\text{Amemiya}}{\Leftrightarrow}\quad
  \norm{x_\alpha\wedge\one}_{L_1}\to 0\quad\Leftrightarrow\quad
  x_\alpha\goesmu 0\text{ in }L_1(\mu)
\end{displaymath}
for every net $(x_\alpha)$ in $X_+$.

\begin{proposition}\label{un-nonat-conv}
  Let $X$ be a non-atomic order continuous Banach lattice and
  $W$ a neighborhood of zero for un-topology. If $W$ is convex then $W=X$.
\end{proposition}

\begin{proof}
  Fix $e\in X_+$; we will show that $e\in W$. We know that
  $V_{\varepsilon,u}\subseteq W$ for some $\varepsilon>0$ and
  $u>0$. Consider the principal band $B_e$. Since $X$ is order
  continuous, $B_e$ is a projection band in $X$; let $P_e$ be the
  corresponding band projection. Furthermore, $B_e$ is a non-atomic
  order continuous Banach lattice with a weak unit. Let
  $L_1(\Omega,\mathcal F,\mu)$ be an AL-representation for $B_e$ with
  $e=\one$. Note that the measure $\mu$ is non-atomic because if a
  measurable set $A$ were an atom for $\mu$ then its characteristic
  function $\chi_A$ would be an atom in $X$. Fix $n\in\mathbb N$.
  Using the non-atomicity of $\mu$, we find a measurable partition
  $A_{n,1},\dots,A_{n,n}$ of $\Omega$ with $\mu(A_{n,i})=\frac1n$ as
  $i=1,\dots,n$; see, e.g., Exercise~2 in~\cite[p.~174]{Halmos:70}.
  Since $L_\infty(\mu)\subseteq B_e\subseteq L_1(\mu)$, we may view
  the characteristic functions $\chi_{A_{n,i}}$ as elements of
  $B_e$. Consider the vectors $(n\chi_{A_{n,i}})\wedge u$ as
  $i=1,\dots,n$; they belong to $B_e$, so that we may view them as
  functions in $L_1(\mu)$. Let $g_n$ be the function in this list
  whose norm in $X$ is maximal; if there are more than one, pick any
  one. Repeating this construction for every $n\in\mathbb N$, we
  produce a sequence $(g_n)$ in $[0,u]\cap B_e$. It follows that
  $g_n\le P_eu$ for every $n$. Since $P_eu$ may be viewed as an
  element of $L_1(\mu)$ and the measure of the support of $g_n$ tends
  to zero, it follows that $\norm{g_n}_{L_1}\to 0$.  Amemiya's Theorem
  yields $\norm{g_n}_X\to 0$. Fix $n$ such that
  $\norm{g_n}_X<\varepsilon$. It follows from the definition of $g_n$
  that $\bignorm{(n\chi_{A_{n,i}})\wedge u}_X<\varepsilon$ as
  $i=1,\dots,n$, so that $n\chi_{A_{n,i}}$ is in $V_{\varepsilon,u}$
  and, therefore, in $W$. Since $W$ is convex and
  \begin{displaymath}
    e=\one=\frac1n\sum_{i=1}^nn\chi_{A_{n,i}},
  \end{displaymath}
  we have $e\in W$. Therefore, $X_+\subseteq W$. Furthermore,
  it follows from $n\chi_{A_{n,i}}\in V_{\varepsilon,u}$ that
  $-n\chi_{A_{n,i}}\in V_{\varepsilon,u}$ for all $i=1,\dots,n$ and,
  therefore, $-e\in W$. This yields 
  $X_-\subseteq W$. Finally, for every $x\in X$ we have
  $x=\frac12\bigl(2x^++2(-x^-)\bigr)$, so that $x\in W$.

\end{proof}

\begin{theorem}\label{loc-conv}
  Let $X$ be an order continuous Banach lattice. Un-topology on $X$ is
  locally convex iff $X$ is atomic.
\end{theorem}

\begin{proof}
  Suppose that $X$ is atomic. By \Cref{un-atomic}, un-topology is
  determined by the family of seminorms $x\mapsto
  \bigabs{f_a(x)}$ where $a$ is an atom of $X$; hence the topology is
  locally convex.

  Suppose that un-topology is locally convex but $X$ is not atomic. It
  follows that there is $e\in X_+$ such that $B_e$ is non-atomic.
  By \Cref{sublat}, un-topology on $B_e$ agrees
  with the relative topology induced on $B_e$ by un-topology on
  $X$; in particular, it is locally convex. On the other hand,
  \Cref{un-nonat-conv} asserts that this topology on $B_e$ has no
  proper convex neighborhoods; a contradiction. 
\end{proof}

\subsection*{Un-continuous functionals}
\Cref{loc-conv} allows us to describe un-continuous linear
functionals. For a functional $\varphi\in X^*$, we say that $\varphi$
is \term{un-continuous} if it is continuous with respect to the
un-topology on $X$ or, equivalently, if $x_\alpha\goesun 0$ implies
$\varphi(x_\alpha)\to 0$.

\begin{proposition}\label{un-cont-ideal}
  The set of all un-continuous functionals in $X^*$ is an ideal.
\end{proposition}

\begin{proof}
  It is straightforward to verify that this set is a linear
  subspace. Suppose that $\varphi$ in $X^*$ is un-continuous; we will
  show that $\abs{\varphi}$ is also un-continuous. Fix $\delta>0$. One
  can find $\varepsilon>0$ and $u>0$ such that
  $\bigabs{\varphi(x)}<\delta$ whenever $x\in V_{\varepsilon,u}$. Fix
  $x\in V_{\varepsilon,u}$.  Since $V_{\varepsilon,u}$ is solid,
  $\abs{y}\le\abs{x}$ implies $y\in V_{\varepsilon,u}$ and, therefore,
  $\bigabs{\varphi(y)}<\delta$. By the Riesz-Kantorovich formula, we
  get
  \begin{displaymath}
    \bigabs{\abs{\varphi}(x)}\le\abs{\varphi}\bigl(\abs{x}\bigr)
    =\sup\bigl\{\bigabs{\varphi(y)}\mid\abs{y}\le\abs{x}\bigr\}
    \le\delta.
  \end{displaymath}
  It follows that $\abs{\varphi}$ is un-continuous. Hence, the set of
  all un-continuous functionals in $X^*$ forms a sublattice. It is
  easy to see that if $\varphi\in X^*_+$ is un-continuous and
  $0\le\psi\le\varphi$ then $\psi$ is also un-continuous; this
  completes the proof.
\end{proof}

Recall that if $a$ is an atom then $f_a$
stands for the corresponding ``coordinate functional''.

\begin{corollary}\label{un-dual}
  Suppose that $X$ is an order continuous Banach lattice and
  $\varphi\in X^*$ is un-continuous.
  \begin{enumerate}
  \item\label{dual-atomic} If $X$ is atomic then
    $\varphi=\lambda_1f_{a_1}+\dots+\lambda_nf_{a_n}$, where
    $\lambda_1,\dots,\lambda_n\in\mathbb R$ and $a_1,\dots,a_n$ are
    atoms;
  \item If $X$ is non-atomic then $\varphi=0$.
  \end{enumerate}
\end{corollary}

\begin{proof}
  By \Cref{un-cont-ideal}, we may assume that  $\varphi\ge 0$;
  otherwise we consider $\varphi^+$ and $\varphi^-$. 

  Suppose $X$ is atomic; let $A$ be a maximal disjoint family of
  atoms. We claim that the set $F:=\{a\in A\mid\varphi(a)\ne 0\}$ is
  finite. Indeed, otherwise, take a sequence $(a_n)$ of distinct atoms
  in $F$ and put $x_n=\frac{1}{\varphi(a_n)}a_n$. Then $x_n\goesun 0$
  by \Cref{un-atomic}, yet $\varphi(x_n)=1$; a contradiction. This
  proves the claim.

  Since $X$ is order continuous, it follows from Remark~\ref{atoms}
  that $X$ has a disjoint band decomposition
  $X=B_F\oplus B_{A\setminus F}$. Since $\varphi(a)=0$ for all
  $a\in A\setminus F$, $\varphi$ vanishes on the ideal
  $I_{A\setminus F}$ and, therefore, on $B_{A\setminus F}$ because
  $\varphi$ is order continuous. On the other hand, since $F$ is
  finite, $B_F=\Span F$ and, therefore, is finite-dimensional. It
  follows that $\varphi$ is a linear combination of
  $\{f_a\mid a\in F\}$.

  Suppose now that $X$ is non-atomic. Let $W=\varphi^{-1}(-1,1)$. Then
  $W$ is a convex neighborhood of zero for the un-topology. By
  \Cref{un-nonat-conv}, $W=X$. This easily implies $\varphi=0$.
\end{proof}

Case~\eqref{dual-atomic} of the preceding corollary essentially says
that every un-continuous functional on an atomic order continuous
space has finite support.

\begin{example}
  Let $X=\ell_2$. By \Cref{un-dual}, the set of all un-continuous
  functionals in $X^*$ may be identified with $c_{00}$, the linear
  subspace of all sequences with finite support. Clearly, it is
  neither norm closed nor order closed; it is not even
  $\sigma$-order closed in $X^*$.
\end{example}

\begin{example}
  Let $X=C_0(\Omega)$ where $\Omega$ is a locally compact Hausdorff
  topological space. It was observed in \cite[Example~20]{Troitsky:04}
  that the un-topology in $X$ agrees with the topology of uniform
  convergence on compact subsets of $\Omega$. 

  Let $\varphi\in X^*_+$. By the Riesz Representation Theorem, there
  exists a regular Borel measure $\mu$ such that $\varphi(f)=\int
  f\,d\mu$ for every $f\in X$; see, e.g.,
  \cite[Theorem~III.5.7]{Conway:90}. An argument similar to the proof
  of  \cite[Proposition~IV.4.1]{Conway:90} shows that $\varphi$ is
  un-continuous iff $\mu$ has compact support.
\end{example}

\section{Un-completeness}

Throughout this section, $X$ is assumed to be an order continuous
Banach lattice. Since un-topology is linear, one can talk about
un-Cauchy nets. That is, a net $(x_\alpha)$ is un-Cauchy if for every
un-neighborhood $U$ of zero there exists $\alpha_0$ such that
$x_\alpha-x_\beta\in U$ whenever
$\alpha,\beta\ge\alpha_0$.
% Equivalently, the ``double net'' of
% differences $(x_\alpha-x_\beta)$ indexed by pairs
% $(\alpha,\beta)\in\Lambda^2$ is un-null.
We investigate whether $X$
itself or some ``nice'' subset of $X$ is un-complete. First, we
observe that the entire space is un-complete only when $X$ is
finite-dimensional.

\begin{lemma}\label{disj-sum}
  Let $(x_n)$ be a positive disjoint sequence in an order continuous
  Banach lattice $X$ such that $(x_n)$ is not norm null. Put
  $s_n=\sum_{i=1}^nx_i$. Then $(s_n)$ is un-Cauchy but not un-convergent.
\end{lemma}

\begin{proof}
  The sequence $(s_n)$ is monotone increasing and does not converge in
  norm; hence it is not un-convergent by
  \Cref{monot}\eqref{monot-un}. To show that $(s_n)$ is un-Cauchy, fix
  any $\varepsilon>0$ and a non-zero $u\in X_+$. Since $x_i$'s are
  disjoint, we have $s_n\wedge u=\sum_{i=1}^n(x_i\wedge u)$. The
  sequence $(s_n\wedge u)$ is increasing and order bounded, hence is
  norm Cauchy by Nakano's Theorem; see
  \cite[Theorem~4.9]{Aliprantis:06}. We can find $n_0$ such that
  \begin{math}
    \bignorm{s_m\wedge u-s_n\wedge u}<\varepsilon
  \end{math}
  whenever $m\ge n\ge n_0$. Observe that
  \begin{displaymath}
    s_m\wedge u-s_n\wedge u=
    \sum_{i=n+1}^m(x_i\wedge u)=(s_m-s_n)\wedge u
    =\abs{s_m-s_n}\wedge u.
  \end{displaymath}
  It follows that $\bignorm{\abs{s_m-s_n}\wedge u}<\varepsilon$, so
  that $s_m-s_n\in V_{\varepsilon,u}$.
\end{proof}

\begin{proposition}
  Let $X$ be an order continuous Banach lattice. $X$ is un-complete
  iff $X$ is finite-dimensional.
\end{proposition}

\begin{proof}
  If $X$ is finite-dimensional then it has a strong unit, so that
  un-topology agrees with norm topology and is, therefore,
  un-complete. Suppose now that $\dim X=\infty$. Then $X$ contains a
  disjoint normalized positive sequence. By \Cref{disj-sum}, $X$ is
  not un-complete.
\end{proof}

\begin{example}
  Let $X=L_p$ with $1<p<\infty$. Pick $0\le x\in L_1\setminus L_p$ and
  put $x_n=x\wedge(n\one)$. It is easy to see that $(x_n)$ is
  un-Cauchy in $L_p$, yet it does not un-converge in $L_p$.
\end{example}

Even when the entire space is not un-complete, the closed unit ball
$B_X$ may still be un-complete; that is, complete in the topology
induced by un-topology on $X$.  Since $B_X$ is
un-closed, it is un-complete iff every norm bounded
un-Cauchy net in $X$ is un-convergent.  The following theorem
should be compared with \cite[Theorem~4.7]{GaoX:14}, where a similar
statement was proved for uo-convergence.

\begin{theorem}\label{BX-complete}
  Let $X$ be an order continuous Banach lattice. Then $B_X$ is
  un-complete iff $X$ is a KB-space.
\end{theorem}

\begin{proof}
  Suppose $X$ is not KB. Then $X$ contains a lattice copy of
  $c_0$. Let $(x_n)$ be the sequence in $X$ corresponding to the unit
  basis of $c_0$. Let $s_n=\sum_{i=1}^nx_i$. Clearly, $(s_n)$ is norm
  bounded. However, by \Cref{disj-sum}, $(s_n)$ is un-Cauchy but not
  un-convergent.

  Suppose now that $X$ is a KB-space. First, we consider the case when
  $X$ has a weak unit. In this case, un-topology on $X$ and,
  therefore, on $B_X$, is metrizable by Theorem~\ref{metriz}. Hence,
  it suffices to prove that $B_X$ is sequentially un-complete. Let
  $(x_n)$ be a sequence in $B_X$ which is un-Cauchy in $X$. Let
  $L_1(\mu)$ be an AL-representation for $X$. It follows that $(x_n)$
  is Cauchy with respect to convergence in measure in $L_1(\mu)$. By
  \cite[Theorem~2.30]{Folland:99}, there is a subsequence $(x_{n_k})$
  which converges a.e.  It follows that $(x_{n_k})$ is uo-Cauchy in
  $X$ by \cite[Remark~4.6]{GTX}. Then \cite[Theorem~4.7]{GaoX:14}
  yields that $x_{n_k}\goesuo x$ for some $x\in X$. It follows that
  $x_{n_k}\goesun x$. 
  Since $(x_n)$ is un-Cauchy, this yields that
  $x_n\goesun x$.

  Now consider the general case. Let $X$ be a KB-space and
  $(x_\alpha)$ a net in $B_X$ such that $(x_\alpha)$ is un-Cauchy in
  $X$; we need to prove that the net is un-convergent. We may assume
  without loss of generality that $x_\alpha\ge 0$ for every $\alpha$;
  otherwise, consider $(x_\alpha^+)$ and $(x_\alpha^-)$, which are
  also un-Cauchy because
  $\abs{x_\alpha^+-x_\beta^+}\le\abs{x_\alpha-x_\beta}$ and 
  $\abs{x_\alpha^--x_\beta^-}\le\abs{x_\alpha-x_\beta}$.
  By \Cref{wunit-decomp}, there exists a dense band decomposition
  $\mathcal B$ of $X$ such that each $B$ in $\mathcal B$ has a
  weak unit. Put
  \begin{displaymath}
    \mathcal C=\bigl\{B_1\oplus\dots\oplus B_n\mid
    B_1,\dots,B_n\in\mathcal B\bigr\}.
  \end{displaymath}
  Note that $\mathcal C$ is a family of bands with weak
  units. Furthermore, $\mathcal C$ is a directed set when ordered by
  inclusion, so the family of band projections
  $(P_C)_{C\in\mathcal C}$ may be viewed as a net.

  For every $C\in\mathcal C$, the net $(P_Cx_\alpha)$ is un-Cauchy by
  \Cref{pb}. Since $C$ has a weak unit, the first part of the proof
  yields that $(P_Cx_\alpha)$ un-converges to some positive vector $x_C$ in
  $C$. This produces a net $(x_C)_{C\in\mathcal C}$. It is easy to
  verify that $x_C=x_{B_1}+\dots+x_{B_n}$ whenever
  $C=B_1\oplus\dots\oplus B_n$ for some $B_1,\dots,B_n\in\mathcal B$.
  It follows that the net $(x_C)_{C\in\mathcal C}$ is increasing. On
  the other hand,
  \begin{math}
    \norm{x_C}\le\liminf_\alpha\norm{P_Cx_\alpha}\le 1,
  \end{math}
  so that this net is norm bounded. Since $X$ is a KB-space, the net
  $(x_C)_{C\in\mathcal C}$ converges in norm to some $x\in X$. 

  Fix $B\in\mathcal B$. On one hand, norm continuity of $P_B$ yields
  $\lim_{C\in\mathcal C}P_Bx_C=P_Bx$. On the other hand, for every
  $C\in\mathcal C$ with $B\subseteq C$ we have $P_Bx_C=x_B$, so that
  $\lim_{C\in\mathcal C}P_Bx_C=x_B$. It follows that $P_Bx=x_B$, so
  that $P_Bx_\alpha\goesun P_Bx$ for every $B\in\mathcal B$. Now
  \Cref{band-decomp} yields $x_\alpha\goesun x$.
\end{proof}

The assumption that $X$ is order continuous cannot be removed: for
example, $\ell_\infty$ is not a KB-space, yet its closed unit ball is
un-complete (because the un and the norm topologies on $\ell_\infty$
agree).

\begin{example}
  The following examples show that in general $B_X$ in
  \Cref{BX-complete} cannot be replaced with an arbitrary convex closed
  bounded set. Let $X=\ell_1$; let $C$ be the set of all vectors in $B_X$ whose
  coordinates sum up to zero. Clearly, $C$ is convex, closed, and
  bounded. Let $x_n=\frac12(e_1-e_n)$. Then $(x_n)$ is a sequence in
  $C$ which un-converges to $\frac12e_1$ which is not in $C$. Thus,
  $C$ is not un-closed in $X$; in particular, $C$ is not un-complete.

  It is easy to construct a similar example in $X=L_1$; take
  \begin{math}
    C=\bigl\{x\in B_X\mid \int x=0\bigr\}
  \end{math}
  and put
  $x_n=\chi_{[0,\frac12]}-\frac{n}{2}\chi_{[\frac12,\frac12+\frac{1}{n}]}$,
  $n\ge 2$.
\end{example}

\begin{proposition}\label{sets-un-closed}
  Suppose that $X^*$ is order continuous and $C$ is a norm closed
  convex norm bounded subset of $X$. Then $C$ is un-closed.
\end{proposition}

\begin{proof}
  Suppose that $x_\alpha\goesun x$ for a net $(x_\alpha)$ in $C$ and a
  vector $x$ in $X$. Since $(x_\alpha)$ is norm bounded and $X^*$ is
  order continuous, \cite[Theorem~6.4]{DOT} guarantees that
  $(x_\alpha)$ converges to $x$ weakly. Since $C$ is convex and
  closed, it is weakly closed, hence $x\in C$.
\end{proof}

\begin{corollary}
  Let $X$ be a reflexive Banach lattice and $C$ a closed convex norm
  bounded subset of $X$. Then $C$ is un-complete.
\end{corollary}

\begin{proof}
  Since $X$ is reflexive, $X$ is a KB-space and $X^*$ is order
  continuous. Let $(x_\alpha)$ be a un-Cauchy net in
  $C$. \Cref{BX-complete} yields that
  $x_\alpha\goesun x$ for some $x\in X$, while \Cref{sets-un-closed}
  implies that $x\in C$.
\end{proof}

\section{Un-compact sets}

The main result of this section is \Cref{BX-compact}, which asserts
that $B_X$ is (sequentially) un-compact iff $X$ is an atomic
KB-space. We start with some auxiliary results.  The following theorem
shows that, under certain assumptions, un-compactness is a ``local''
property.

\begin{theorem}\label{comp-local}
  Let $X$ be a KB-space, $\mathcal B$ a dense band decomposition of
  $X$, and $A$ a un-closed norm bounded subset of $X$. Then $A$ is
  un-compact iff $P_B(A)$ is un-compact in $B$ for every $B\in\mathcal B$. 
\end{theorem}

\begin{proof}
  If $A$ is un-compact then $P_B(A)$ is un-compact in $B$ for every
  $B\in\mathcal B$ because $P_B$ is un-continuous by \Cref{pb}. To
  prove the converse, suppose that $P_B(A)$ is un-compact in $B$ for
  every $B\in\mathcal B$. Let $H=\prod_{B\in\mathcal B}B$, the formal
  product of all the bands in $\mathcal B$. That is, $H$ consists of
  families $(x_B)_{B\in\mathcal B}$ indexed by $\mathcal B$, where
  $x_B\in B$. We equip $H$ with the topology of coordinate-wise
  un-convergence; this is the product of un-topologies on the bands
  that make up $H$. This makes $H$ a topological vector space. Define
  $\Phi\colon X\to H$ via $\Phi(x)=(P_Bx)_{B\in\mathcal B}$. Clearly,
  $\Phi$ is linear. Since $\mathcal B$ is a dense band decomposition,
  $\Phi$ is one-to-one. By \Cref{band-decomp}, $\Phi$ is a
  homeomorphism from $X$ equipped with un-topology onto its range in
  $H$.
 
  Let $K$ be the subset of $H$ defined by $K=\prod_{B\in\mathcal
    B}P_B(A)$. By Tikhonov's Theorem, $K$ is compact in $H$. It is
  easy to see that $\Phi(A)\subseteq K$.

  We claim that $\Phi(A)$ is closed in $H$. Indeed, suppose that
  $\Phi(x_\alpha)\to h$ in $H$ for some net $(x_\alpha)$ in $A$. In
  particular, the net $\bigl(\Phi(x_\alpha)\bigr)$ is Cauchy in
  $H$. Since $\Phi$ is a homeomorphism, the net $(x_\alpha)$ is
  un-Cauchy in $A$. Since $(x_\alpha)$ is bounded and $X$ is a
  KB-space, $(x_\alpha)$ un-converges to some $x\in X$ by
  \Cref{BX-complete}. Since $A$ is un-closed, we have $x\in A$. It
  follows that $h=\Phi(x)$, so that $h\in\Phi(A)$. 

  Being a closed subset of a compact set, $\Phi(A)$ is itself
  compact. Since $\Phi$ is a homeomorphism, we conclude that $A$ is
  un-compact.
\end{proof}

Next, we discuss relationships between the sequential and the general
variants of un-closedness and un-compactness. Recall that for a set
$A$ in a topological space, we write $\overline{A}$ for the
closure of $A$; we write $\overline{A}^\sigma$ for the
\term{sequential closure} of $A$, i.e., $a\in\overline{A}^\sigma$ iff
$a$ is the limit of a sequence in $A$. We say that $A$ is
\term{sequentially closed} if $\overline{A}^\sigma=A$. It is well
known that for a metrizable topology, we always have
$\overline{A}^\sigma=\overline{A}$.

For a set $A$ in a Banach lattice, we write $\overline{A}^{\rm un}$
and $\overline{A}^{\sigma\text{-un}}$ for the un-closure and the
sequential un-closure of $A$, respectively. Obviously, 
$\overline{A}^{\sigma\text{-un}}\subseteq\overline{A}^{\rm
  un}$.

\begin{example}
  In general, $\overline{A}^{\rm
    un}\ne\overline{A}^{\sigma\text{-}\rm{un}}$. Indeed, in the
  notation of \Cref{emelyanov}, let
  $A=\{e_{\omega}\mid\omega\in\Omega\}$. It follows from
  \Cref{emelyanov} that zero is in $\overline{A}^{\rm
  un}$ but not in $\overline{A}^{\sigma\text{-un}}$.
\end{example}

\begin{proposition}\label{closures}
  Let $A$ be a subset of a Banach lattice $X$. If $X$ has a
  quasi-interior point \emph{or} $X$ is order continuous then 
  $\overline{A}^{\rm un}=\overline{A}^{\sigma\text{-}\rm{un}}$.
\end{proposition}

\begin{proof}
  If $X$ has a quasi-interior point then its un-topology is metrizable
  by \Cref{metriz}, hence $\overline{A}^{\rm
    un}=\overline{A}^{\sigma\text{-}\rm{un}}$.

  Suppose that $X$ is order continuous. Suppose that
  $x\in\overline{A}^{\rm un}$; we need to show that
  $x\in\overline{A}^{\sigma\text{-un}}$. Without loss of generality,
  $x=0$. This means that $A$ contains a un-null net $(x_\alpha)$. By
  \Cref{KP}, there exists an increasing sequence of indices
  $(\alpha_k)$ and a disjoint sequence $(d_k)$ such that
  $x_{\alpha_k}-d_k\goesnorm 0$. It follows that
  $x_{\alpha_k}-d_k\goesun 0$. Since $(d_k)$ is disjoint, it is
  uo-null and, since $X$ is order continuous, un-null. It follows that
  $x_{\alpha_k}\goesun 0$ and, therefore,
  $0\in\overline{A}^{\sigma\text{-un}}$.
\end{proof}

Recall that a topological space is said to be \term{sequentially
  compact} if every sequence has a convergent subsequence. In a
Hausdorff topological vector space which is metrizable (or,
equivalently, first countable), sequential compactness is equivalent
to compactness, see, e.g., \cite[Theorem~7.21]{Royden:88}. We do not
know whether un-compactness and sequential un-compactness are
equivalent in general, yet we have the following partial result.

\begin{proposition}\label{scomp}
  Let $A$ be a subset of a Banach lattice $X$.
  \begin{enumerate}
  \item\label{scomp-qip} If $X$ has a quasi-interior point, then $A$
    is sequentially un-compact iff $A$ is un-compact.
  \item\label{scomp-oc} Suppose that $X$ is order continuous. If $A$ is
    un-compact then $A$ is sequentially un-compact.
  \item\label{scomp-KB} Suppose that $X$ is a KB-space. If $A$ is
    norm bounded and sequentially un-compact then $A$ is un-compact.
  \end{enumerate}
\end{proposition}

\begin{proof}
  \eqref{scomp-qip} follows immediately from \Cref{metriz}.

  \eqref{scomp-oc} Let $(x_n)$ be a sequence in $A$. Find $e\in X_+$
  such that $(x_n)$ is contained in $B_e$ (e.g., take
  $e=\sum_{n=1}^\infty\frac{x_n}{2^n\norm{x_n}+1}$). Since $B_e$ is
  un-closed, the set $A\cap B_e$ is un-compact in $B_e$. Since $e$ is
  a quasi-interior point for $B_e$, the un-topology on $B_e$ is
  metrizable, hence $A\cap B_e$ is sequentially un-compact. It follows
  that there is a subsequence $(x_{n_k})$ which un-converges in
  $B_e$ to some $x\in A\cap B_e$. By
  \Cref{sublat}\eqref{sublat-projb}, $x_{n_k}\goesun x$ in $X$.

  \eqref{scomp-KB} Clearly, $A$ is sequentially un-closed and,
  therefore, un-closed by \Cref{closures}.  Let $\mathcal B$ be as in
  \Cref{wunit-decomp}. For each $B\in\mathcal B$, the band projection
  $P_B$ is un-continuous by \Cref{pb}, so that $P_B(A)$ is
  sequentially un-compact in $B$. Since $B$ has a weak unit, the
  un-topology on $B$ is metrizable, so that $P_B(A)$ is un-compact in
  $B$. The conclusion now follows from \Cref{comp-local}.
\end{proof}

\begin{theorem}\label{BX-compact}
  For a Banach lattice $X$, TFAE:
  \begin{enumerate}
  \item\label{BX-compact-un} $B_X$ is un-compact;
  \item\label{BX-compact-seq} $B_X$ is sequentially un-compact;
  \item\label{BX-compact-KB} $X$ is an atomic KB-space.
  \end{enumerate}
\end{theorem}

\begin{proof}
  First, observe that both \eqref{BX-compact-un} and
  \eqref{BX-compact-seq} imply that $X$ is order continuous and
  atomic. Indeed, since order intervals are bounded and un-closed,
  they are (sequentially) un-compact. But on order intervals, the
  un-topology agrees with the norm topology, hence order intervals are
  norm compact. This implies that $X$ is atomic and order
  continuous; see, e.g., \cite[Theorem~6.1]{Wnuk:99}.

  Suppose \eqref{BX-compact-un}. Since $X$ is order continuous,
  \Cref{scomp}\eqref{scomp-oc} yields~\eqref{BX-compact-seq}.

  Suppose \eqref{BX-compact-seq}. We already know that $X$ is
  atomic. To show that $X$ is a KB-space, let $(x_n)$ be an increasing
  norm bounded sequence in $X_+$. By assumption, it has a
  un-convergent subsequence $(x_{n_k})$. By
  \Cref{monot}\eqref{monot-un},  $(x_{n_k})$ converges in norm, hence
  $(x_n)$ converges in norm. This yields~\eqref{BX-compact-KB}.

  Suppose \eqref{BX-compact-KB}. Let $A$ be a maximal disjoint family
  of atoms in $X$. Then $\bigl\{B_a\mid a\in A\bigr\}$ is a dense band
  decomposition of $X$. For every $a\in A$, $P_a(B_X)$ is a closed
  bounded subset of the one-dimensional band $B_a$, hence $P_a(B_X)$
  is norm and un-compact in $B_a$. \Cref{comp-local} now implies that
  $B_X$ is un-compact, which yields~\eqref{BX-compact-un}.
\end{proof}

\begin{example}
  Let $X=c_0$ and $x_n=e_1+\dots+e_n$. Then $(x_n)$ is a sequence in
  $B_X$ with no un-convergent subsequences.
\end{example}

\begin{proposition}\label{rcomp}
  Let $A$ be a subset of an order continuous Banach lattice $X$. If
  $A$ is relatively un-compact then $A$ is relatively sequentially
  un-compact.
\end{proposition}

\begin{proof}
  Let $(x_n)$ be a sequence in $A$. Find $e\in X_+$ such that $(x_n)$
  is contained in $B_e$. Since $\overline{A}^{\rm un}$ is un-compact,
  the set $\overline{A}^{\rm un}\cap B_e$ is un-compact in $B_e$ and,
  therefore, sequentially un-compact in $B_e$ because the un-topology
  on $B_e$ is metrizable. Hence, there is a subsequence $(x_{n_k})$
  which un-converges in $B_e$ and, therefore, in $X$.
\end{proof}

\section{Un-convergence and weak*-convergence}

\subsection*{When does un-convergence imply weak*-convergence?}
It is easy to see that, in general, un-convergence does not imply
weak*-convergence. Indeed, let $X$ be an infinite-dimensional Banach
lattice with order continuous dual. Pick any unbounded disjoint
sequence $(f_n)$ in $X^*$. Being unbounded, $(f_n)$ cannot be
weak*-null. Yet it is un-null by \Cref{disj}. However, if we restrict
ourselves to norm bounded nets, the situation is more interesting. The
following result is analogous to \cite[Theorem~2.1]{Gao:14}.  Recall
that for a net $(f_\alpha)$ in $X^*$, we write
$f_{\alpha}\xrightarrow{\abs{\sigma}(X^*,X)}0$ if
$\abs{f_\alpha}(x)\to 0$ for every $x\in X_+$.

\begin{theorem}\label{un-ws}
  Let $X$ be a Banach lattice such that $X^*$ is order continuous. The
  following are equivalent:
  \begin{enumerate}
  \item\label{un-ws-oc} $X$ is order continuous;
  \item\label{un-ws-net} for any norm bounded net $(f_{\alpha})$ in $X^*$,
    if $f_{\alpha}\goesun 0$, then $f_{\alpha}\goesws 0$;
  \item\label{un-ws-net-abs} for any norm bounded net $(f_{\alpha})$ in $X^*$,
    if $f_{\alpha}\goesun 0$, then
    $f_{\alpha}\xrightarrow{\abs{\sigma}(X^*,X)}0$;
  \item\label{un-ws-seq} for any norm bounded sequence $(f_n)$ in $X^*$,
    if $f_n\goesun 0$, then $f_n\goesws 0$;
  \item\label{un-ws-seq-abs} for any norm bounded sequence $(f_n)$ in $X^*$,
    if $f_n\goesun 0$, then
    $f_n\xrightarrow{\abs{\sigma}(X^*,X)}0$.
  \end{enumerate}
\end{theorem}

The proof is similar to that of \cite[Theorem~2.1]{Gao:14} except that
in the proof of \eqref{un-ws-seq}$\Rightarrow$\eqref{un-ws-oc} we use
\Cref{disj}. Note that without the assumption that $X^*$ is
order continuous, we still get the following implications:
\begin{displaymath}
  \eqref{un-ws-oc}\Rightarrow
  \bigl[\eqref{un-ws-net}\Leftrightarrow\eqref{un-ws-net-abs}\bigr]
  \Rightarrow
  \bigl[\eqref{un-ws-seq}\Leftrightarrow\eqref{un-ws-seq-abs}\bigr].
\end{displaymath}

\subsection*{When does weak*-convergence imply un-convergence?}
Recall that for norm bounded nets, weak*-convergence
implies uo-convergence in $X^*$ iff $X$ is atomic and order continuous
by \cite[Theorem~3.4]{Gao:14}. Furthermore, \Cref{w-un} immediately
yields the following.

\begin{corollary}\label{ws-un-gen}
  If $f_n\goesws 0$ implies $f_n\goesun 0$ for every sequence in $X^*$
  then $X^*$ is atomic and order continuous.
\end{corollary}

The following example shows that the converse is false in general.

\begin{example}
  Let $X=c$, the space of all convergent sequences. By \cite[Theorem
  16.14]{Aliprantis:06a}, $X^*$ may be identified with
  $\ell_1\oplus\mathbb R$ with the duality given by
  \begin{displaymath}
    \bigl\langle(f,r),x\bigr\rangle=r\cdot\lim_nx_n+\sum_{n=1}^\infty f_nx_n,
  \end{displaymath}
  where $x\in c$, $f\in\ell_1$, and $r\in\mathbb R$. It is easy to see
  that $X^*$ is atomic and order continuous. Consider the sequence
  $\bigl((e_n,0)\bigr)$ in $X^*$, where $e_n$ is the $n$-th standard
  unit vector in $\ell_1$. It is easy to see that
  $(e_n,0)\goesws(0,1)$ in $X^*$. On the other hand, this sequence is
  disjoint and, therefore, un-null. Take $f_n=(e_n,-1)$; it follows
  that $(f_n)$ is weak*-null but not un-null. Note that in this
  example, $X^*$ is order continuous while $X$ is not.
\end{example}

Nevertheless, we will show that the converse implication is true under
the additional assumption that $X$ is order continuous.
%\cite[Theorem~6.1]{Wnuk:99} asserts that if $X^*$ is atomic and order
%continuous then $X$ is $\sigma$-order complete iff $X$ is order
%continuous.

% \begin{lemma}\label{oc-star}
%   Suppose that $X^*$ is order continuous and $f_\alpha\goesws 0$
%   implies $f_\alpha\goesun 0$ for every net $(f_\alpha)$
%   in $X^*$. Then $X$ is order continuous.
% \end{lemma}

% \begin{proof}
%   Suppose not. Then by \cite[Corollary 2.4.3]{Meyer-Nieberg:91} there
%   exists a disjoint norm-bounded sequence $(f_n)$ in $X^*$ which is
%   not weak*-null. One can then find a subsequence $(f_{n_k})$, a
%   vector $x_0\in X$ and a positive real $\varepsilon$ so that
%   $\bigabs{f_{n_k}(x_0)}>\varepsilon$ for every $k$. By Alaoglu-Bourbaki
%   Theorem, there is a subnet $(g_\alpha)$ of $(f_{n_k})$ such that
%   $g_\alpha\goesws g$ for some $g\in X^*$. Since $(f_{n_k})$ is
%   disjoint and $X^*$ is order continuous, $f_{n_k}\goesun 0$ and,
%   therefore, $g_\alpha\goesun 0$. By assumption, this yields $g=0$, so
%   that $g_\alpha\goesws 0$. This contradicts
%   $\abs{g_\alpha(x_0)}>\varepsilon$ for every $\alpha$.
% \end{proof}

\begin{theorem}\label{ws-un}
  The following are equivalent:
  \begin{enumerate}
  \item\label{ws-un-net} For every net $(f_\alpha)$ in $X^*$, if
    $f_\alpha\goesws 0$ then $f_\alpha\goesun 0$;
  \item\label{ws-un-X} $X^*$ is atomic and both $X$ and $X^*$ are
    order continuous.
  \end{enumerate}
\end{theorem}

\begin{proof}
  \eqref{ws-un-net}$\Rightarrow$\eqref{ws-un-X} By \Cref{ws-un-gen},
  $X^*$ is atomic and order continuous. Suppose $X$ is
  not order continuous. By \cite[Corollary 2.4.3]{Meyer-Nieberg:91}
  there exists a disjoint norm-bounded sequence $(f_n)$ in $X^*$ which
  is not weak*-null. One can then find a subsequence $(f_{n_k})$, a
  vector $x_0\in X$ and a positive real $\varepsilon$ so that
  $\bigabs{f_{n_k}(x_0)}>\varepsilon$ for every $k$. By
  the Alaoglu-Bourbaki Theorem, there is a subnet $(g_\alpha)$ of
  $(f_{n_k})$ such that $g_\alpha\goesws g$ for some $g\in X^*$. Since
  $(f_{n_k})$ is disjoint and $X^*$ is order continuous, we have
  $f_{n_k}\goesun 0$ and, therefore, $g_\alpha\goesun 0$. By
  assumption, this yields $g=0$, so that $g_\alpha\goesws 0$. This
  contradicts $\abs{g_\alpha(x_0)}>\varepsilon$ for every $\alpha$.

  \eqref{ws-un-X}$\Rightarrow$\eqref{ws-un-net} 
  Let $f_\alpha\goesws 0$ in $X$. Let $A$ be a maximal disjoint
  collection of atoms in $X^*$; for each atom $a\in A$ let $P_a$ and
  $\varphi_a$ be the corresponding band projection and the coordinate
  functional, respectively; $P_a$ and $\varphi_a$ are defined on
  $X^*$. By \cite[Corollary~2.4.7]{Meyer-Nieberg:91}, $P_a$ and,
  therefore, $\varphi_a$, is weak*-continuous. It follows that
  $\varphi_a(f_\alpha)\to 0$ in $\alpha$. \Cref{un-atomic} yields that
  $f_\alpha\goesun 0$.
\end{proof}

\begin{proposition}
  Suppose that $X^*$ is atomic. The following are equivalent.
  \begin{enumerate}
  \item\label{aws-un-net} For every net $(f_\alpha)$ in $X^*$, if
  \begin{math}
    f_\alpha\xrightarrow{\abs{\sigma}(X^*,X)}0
  \end{math}
  then $f_\alpha\goesun 0$;
  \item\label{aws-un-seq} For every sequence $(f_n)$ in $X^*$, if
  \begin{math}
    f_n\xrightarrow{\abs{\sigma}(X^*,X)}0
  \end{math}
  then $f_n\goesun 0$;
  \item\label{aws-un-X} $X^*$ is order continuous.
  \end{enumerate}
\end{proposition}

\begin{proof}
  \eqref{aws-un-net}$\Rightarrow$\eqref{aws-un-seq} is trivial.

  \eqref{aws-un-seq}$\Rightarrow$\eqref{aws-un-X} The proof is similar
  to that of \Cref{w-un}. To show that $X^*$
  is order continuous, suppose that $(f_n)$ is an order bounded
  positive disjoint sequence in $X_+^*$. It follows that 
  \begin{math}
    f_n\xrightarrow{\abs{\sigma}(X^*,X)}0
  \end{math}
  and, by assumption, $f_n\goesun 0$. Since the sequence is order
  bounded, this yields $f_n\goesnorm 0$. Therefore, $X^*$ is order
  continuous.

  \eqref{aws-un-X}$\Rightarrow$\eqref{aws-un-net} By
  \cite[Proposition~2.4.5]{Meyer-Nieberg:91}, band projections on
  $X^*$ are $\abs{\sigma}(X^*,X)$-continuous. The proof is now
  analogous to the implication
  \eqref{ws-un-X}$\Rightarrow$\eqref{ws-un-net} in \Cref{ws-un}.
\end{proof}

\subsection*{Simultaneous weak* and un-convergence}
Section~4 of~\cite{Gao:14} contains several results that assert that
if a sequence or a net in $X^*$ converges in both weak* and
uo-topology then it also converges in some other topology. Several of
these results remain valid if uo-convergence is replaced with
un-convergence. In particular, this works for Proposition~4.1
in~\cite{Gao:14}. Propositions~4.3, 4.4, and~4.6 in~\cite{Gao:14}
remain valid under the additional assumption that $X^*$ is order
continuous (note that the dual positive Schur property already implies
that $X^*$ is order continuous by
\cite[Proposition~2.1]{Wnuk:13}). The proofs are analogous to the
corresponding proofs in~\cite{Gao:14}. Alternatively, the un-versions
of these may be deduced from the uo-versions using the following two
facts: first, every un-convergent sequence has a uo-convergent
subsequence and, second, a sequence $(x_n)$ converges to $x$ in a
topology $\tau$ iff every subsequence $(x_{n_k})$ has a further
subsequence $(x_{n_{k_i}})$ such that
$x_{n_{k_i}}\xrightarrow{\tau}x$.

\section{Un-compact operators}

Throughout this section, let $E$ be a Banach space, $X$ a Banach
lattice, and $T\in L(E,X)$. We say that $T$ is \term{(sequentially)
un-compact} if $TB_E$ is relatively (sequentially) un-compact in
$E$. Equivalently, for every bounded net $(x_\alpha)$ (respectively,
every bounded sequence $(x_n)$) its image has a subnet
(respectively, subsequence), which is un-convergent.

Clearly, if $T$ is compact then it is un-compact and sequentially
un-compact. \Cref{metriz,BX-compact} and \Cref{rcomp} yield the following.

\begin{proposition}\label{un-comp-op}
  Let $T\in L(E,X)$.
  \begin{enumerate}
  \item If $X$ has a quasi-interior point then $T$ is un-compact iff
  it is sequentially un-compact;
  \item If $X$ is order continuous and $T$ is un-compact
  then $T$ is sequentially un-compact;
  \item\label{un-comp-op-atom} If $X$ is an atomic KB-space then $T$ is
  un-compact and sequentially un-compact.
  \end{enumerate}
\end{proposition}

\begin{proposition}
  The set of all un-compact operators is a linear subspace of
  $L(E,X)$. The set of all sequentially un-compact operators in
  $L(E,X)$ is a closed subspace.
\end{proposition}

\begin{proof}
  Linearity is straightforward. To prove closedness, suppose that
  $(T_m)$ is a sequence of sequentially un-compact operators in $L(E,X)$
  and $T_m\goesnorm T$. We will show that $T$ is sequentially
  un-compact.

  Let $(x_n)$ be a sequence in $B_E$. For every $m$, the sequence
  $(T_mx_n)_n$ has a un-convergent subsequence. By a standard diagonal
  argument, we can find a common subsequence for all these
  sequences. Passing to this subsequence, we may assume without loss
  of generality that for every $m$ we have $T_mx_n\goesun y_m$ for
  some $y_m$. Note that
  \begin{displaymath}
    \norm{y_m-y_k}\le\liminf_n\norm{T_mx_n-T_kx_n}\le\norm{T_m-T_k}\to 0,
  \end{displaymath}
  so that the sequence $(y_m)$ is Cauchy and, therefore, $y_m\goesnorm
  y$ for some $y\in X$.

  Fix $u\in X_+$ and $\varepsilon>0$. Find $m_0$ such that
  $\norm{T_{m_0}-T}<\varepsilon$ and
  $\norm{y_{m_0}-y}<\varepsilon$. Find $n_0$ such that
  \begin{math}
    \bignorm{\abs{T_{m_0}x_n-y_{m_0}}\wedge u}<\varepsilon
  \end{math}
  whenever $n\ge n_0$. It follows from
  \begin{displaymath}
    \abs{Tx_n-y}\wedge u\le\abs{Tx_n-T_{m_0}x_n}+
     \abs{T_{m_0}x_n-y_{m_0}}\wedge u+\abs{y_{m_0}-y}
  \end{displaymath}
  that
  \begin{math}
    \bignorm{\abs{Tx_n-y}\wedge u}<3\varepsilon,
  \end{math}
  so that $Tx_n\goesun y$.
\end{proof}

We do not know whether the set of all un-compact operators is closed.

It is easy to see that if we multiply a (sequentially) un-compact
operator by another bounded operator on the right, the product is
again (sequentially) un-compact. The following example shows that this
fails when we multiply on the left.

\begin{example}\label{Radem}
  \emph{The class of all (sequentially) un-compact operators is not a
    left ideal.} Let $T\colon\ell_1\to L_1$ be defined via
  $Te_n=r_n^+$, where $(e_n)$ is the standard unit basis of $\ell_1$
  and $(r_n)$ is the Rademacher sequence in $L_1$. Note that $T$ is
  neither un-compact nor sequentially un-compact because the sequence
  $(Te_n)$ has no un-convergent subsequences. On the other hand,
  $T=TI_{\ell_1}$, where $I_{\ell_1}$ is the identity operator on
  $\ell_1$. Observe that $I_{\ell_1}$ is un-compact by
  \Cref{un-comp-op}\eqref{un-comp-op-atom}.
\end{example}

\begin{proposition}
  In the diagram $E\xrightarrow{T}X\xrightarrow{S}Y$, suppose that $T$
  is (sequentially) un-compact and $S$ is a lattice homomorphism. If the
  ideal generated by $\Range S$ is dense in $Y$ then $ST$ is
  (sequentially) un-compact.
\end{proposition}

\begin{proof}
  We will prove the statement for the sequential case; the other case
  is analogous. Let $(h_n)$ be a norm bounded sequence in $E$. By
  assumption, there is a subsequence $(h_{n_k})$ such that
  $Th_{n_k}\goesun x$ for some $x\in X$. Let $Z=\Range S$; it is a
  sublattice of $Y$. Fix $u\in Z_+$. Then $u=Sv$ for some $v\in
  X_+$, and $\abs{Th_{n_k}-x}\wedge v\goesnorm 0$. Applying $S$, we
  get $\bigabs{STh_{n_k}-Sy}\wedge u\goesnorm 0$. Therefore,
  $STh_{n_k}\goesun Sx$ in $Z$. It follows from
  \Cref{sublat}\eqref{sublat-maj} and~\eqref{sublat-dense} that
  $STh_{n_k}\goesun Sx$ in $Y$.
\end{proof}

\begin{example}
  \emph{The set of all sequentially un-compact operators is not order
    closed.} Let $T$ be as in \Cref{Radem}. Let $T_n=TP_n$, where
  $P_n$ is the $n$-th basis projection on $\ell_1$, i.e.,
  $T_nh=\sum_{i=1}^nh_ir_i^+$ for $h\in\ell_1$. It is easy to see that
  each $T_n$ is finite rank and, therefore, sequentially
  un-compact. Note that $T_n\uparrow T$, yet $T$ is not sequentially
  un-compact.
\end{example}

\begin{proposition}
  Suppose that for every sequence $(T_n)$ of sequentially un-compact
  operators in $L(c_0,X)$, $T_n\uparrow T$ implies that $T$ is
  sequentially un-compact. Then $X$ is a KB-space.
\end{proposition}

\begin{proof}
  Suppose not. Then there is a lattice isomorphism $T\colon c_0\to
  X$. Put $x_n=Te_n$, where $(e_n)$ is the standard unit basis of
  $c_0$. Put $T_n=TP_n$, where $P_n$ is the $n$-th basis projection on
  $c_0$, i.e., $T_nh=\sum_{i=1}^nh_ix_i$ for $h\in c_0$. It follows
  that $T_nh\goesnorm Th$, so that $T_nh\uparrow Th$ for every $h\ge
  0$ and, therefore, $T_n\uparrow T$. For each $n$, $T_n$ has finite
  rank and, therefore, is sequentially un-compact.

  We claim that, nevertheless, $T$ is not sequentially un-compact. Put
  $w_n=e_1+\dots+e_n$ in $c_0$. Note that $(w_n)$ is norm bounded and
  $Tw_n=x_1+\dots+x_n$. Since $T$ is an isomorphism, $(Tw_n)$ is not
  norm-convergent. Since $(Tw_n)$ is increasing, it is not
  un-convergent by \Cref{monot}\eqref{monot-un}. Similarly, no
  subsequence of $(Tw_n)$ is un-convergent.
\end{proof}

We do not know whether the converse is true.

Next, we study whether un-compactness is inherited under
domination. The following example shows that, in general, the answer
is negative.

\begin{example}
  Let $T$ be as in \Cref{Radem}. Let $S\colon\ell_1\to L_1$ be defined
  via $Se_n=\one$. Then $S$ is a rank-one operator; hence it is
  compact and un-compact. Clearly, $0\le T\le S$. Yet $T$ is not
  un-compact.
\end{example}

\begin{proposition}
  Suppose that $S,T\colon E\to X$, $0\le S\le T$, $X$ is a KB-space
  and $T$ is a lattice homomorphism. If $T$ is (sequentially)
  un-compact then so is $S$.
\end{proposition}

\begin{proof}
 We will prove the sequential case; the other case is similar.
  Let $(h_n)$ be a bounded sequence in $E$. Passing to a subsequence,
  we may assume that $(Th_n)$ is un-convergent. In particular, it is
  un-Cauchy. Fix $u\in X_+$. Note that
  \begin{displaymath}
    \abs{Sh_n-Sh_m}\wedge u
    \le\bigl(S\abs{h_n-h_m}\bigr)\wedge u
    \le\bigl(T\abs{h_n-h_m}\bigr)\wedge u
    =\abs{Th_n-Th_m}\wedge u\goesnorm 0
  \end{displaymath}
  as $n,m\to\infty$. It follows that $(Sh_n)$ is un-Cauchy and,
  therefore, un-converges by \Cref{BX-complete}.
\end{proof}

We would like to mention that the class of un-compact operators is
different from several other known classes of operators. We already
mentioned that every compact operator is un-compact. The converse is
false as the identity operator on any infinite-dimensional atomic
KB-space is un-compact but not compact.

Recall that an operator between Banach lattices is AM-compact if it
maps order intervals to relatively compact sets.

\begin{proposition}
  Every order bounded un-compact operator is AM-compact.
\end{proposition}

\begin{proof}
  Let $T\colon X\to Y$ be an order bounded un-compact operator between
  Banach lattices. Fix an order interval $[a,b]$ in $X$. Since $T$ is
  un-compact, $T[a,b]\subseteq C$ for some un-compact set $C$. Since
  $T$ is order bounded, $T[a,b]\subseteq[c,d]$ for some $c,d\in
  Y$. Note that $[c,d]$ is un-closed, hence $C\cap[c,d]$ is un-compact
  and, being order bounded, is compact. It follows that $T[a,b]$ is
  relatively compact.
\end{proof}

Note that the converse is false: the identity operator on $c_0$ is
AM-compact but not un-compact.

The identity operator on $\ell_1$ is un-compact, yet it is neither
L-weakly compact nor M-weakly compact.

Finally, we note that if $T$ is sequentially un-compact and
semi-compact then $T$ is compact. Indeed, let $(h_n)$ be a bounded
sequence in $E$. There is a subsequence $(h_{n_k})$ such that
$Th_{n_k}\goesun x$ for some $x\in X$. Since $T$ is semi-compact, the
sequence $(Th_{n_k})$ is almost order bounded and, therefore,
$Th_{n_k}\goesnorm x$ by \cite[Lemma~2.9]{DOT}.

Finally, we discuss when weakly compact operators are un-compact.

\begin{lemma}\label{w-un-lim}
  If $x_n\goesw x$ and $x_n\goesun y$ then $x=y$.
\end{lemma}

\begin{proof}
  Without loss of generality, $y=0$. By \Cref{KP}, there
  exist a subsequence $(x_{n_k})$ and a disjoint sequence $(d_k)$
  such that $x_{n_k}-d_k\goesnorm 0$. It follows that
  $x_{n_k}-d_k\goesw 0$, so that $d_k\goesw x$. Now
  \cite[Theorem~4.34]{Aliprantis:06} yields $x=0$.
\end{proof}

\begin{theorem}
  A Banach lattice $X$ is atomic and order continuous iff $T$ is
  sequentially un-compact for every Banach space $E$ and every weakly
  compact operator $T\colon E\to X$.
\end{theorem}

\begin{proof}
  The forward implication follows immediately from \Cref{w-un}.  To
  prove the converse, let $(x_n)$ be a weakly null sequence in $X$. By
  \Cref{w-un}, it suffices to show that $x_n\goesun 0$. Define
  $T\colon\ell_1\to X$ via $Te_n=x_n$. By
  \cite[Theorem~5.26]{Aliprantis:06}, $T$ is weakly compact. By
  assumption, $T$ is sequentially un-compact. It follows that $(Te_n)$
  has a un-convergent subsequence, i.e., $x_{n_k}\goesun x$ for some
  $x\in X$ and a subsequence $(x_{n_k})$. \Cref{w-un-lim} yields
  $x=0$. By the same argument, every subsequence of $(x_n)$ has a
  further subsequence which is un-null; since un-convergence is
  topological, it follows that $x_n\goesun 0$.
\end{proof}

\begin{corollary}
  Every operator from a reflexive Banach space to an atomic order
  continuous Banach lattice is sequentially un-compact.
\end{corollary}

\bigskip

\subsection*{Acknowledgement and further remarks.}
Most of the work on this paper was done during a visit of the first
and the second author to the University of Alberta. After the work on
this paper was essentially completed, we learned of recent preprints
\cite{Zab,GLX}. In the former, the author studies \emph{unbounded
  absolute weak} convergence; it is shown there that in certain
situations it agrees with un-convergence. In the latter, techniques
of unbounded convergence are used to study risk measures.

The authors would like to thank the reviewers for valuable comments
and improvements.


\begin{thebibliography}{WWWW}

\bibitem[AA02]{Abramovich:02}
Y.~Abramovich and C.D.~Aliprantis,
\emph{An invitation to operator theory},
Vol. 50. Providence, RI: American Mathematical Society, 2002.

\bibitem[AB06]{Aliprantis:06}
C.D.~Aliprantis and O.~Burkinshaw,
\emph{Positive operators}, 2nd edition, Springer 2006.

\bibitem[AB06a]{Aliprantis:06a}
C.D.~Aliprantis, K. C. Border,
\emph{Infinite Dimensional Analysis: A hitchhiker's guide},
3rd edition, Springer, Berlin, 2006.

\bibitem[CW98]{Chen:98}
Z.L.~Chen and A.W.~Wickstead,
Relative weak compactness of solid hulls in Banach lattices,
\emph{Indag.\ Math.}, 9(2), 1998, 187--196.

\bibitem[Con99]{Conway:90}
J.B. Conway,
\emph{A course in functional analysis},
2nd edition, Springer-Verlag, New York, 1990

\bibitem[DeM64]{DeMarr:64}
R.~DeMarr,
Partially ordered linear spaces and locally convex linear topological
spaces,
\emph{Illinois J.\ Math.} 8, 1964, 601-–606.
    
\bibitem[DOT]{DOT}
Y.~Deng, M.~O'Brien, and V.G.~Troitsky, 
Unbounded norm convergence in Banach lattices,
\emph{Positivity}, to appear. doi:10.1007/s11117-016-0446-9.

\bibitem[EM16]{Emelyanov:16}
E.Yu.~Emelyanov, M.A.A.~Marabeh,
Two measure-free versions of the Brezis-Lieb Lemma,
\emph{Vladikavkaz Math.\ J.}, 18(1), 2016, 21--25.

\bibitem[Fol99]{Folland:99}
G.B.~Folland,
\emph{Real analysis: Modern techniques and their applications},
2nd edition, Pure and Applied Mathematics,
John Wiley \& Sons, Inc., New York, 1999.

\bibitem[GTX]{GTX}
N.~Gao, V.G.~Troitsky, and F.~Xanthos,
Uo-convergence and its applications to Ces\`aro means in Banach lattices,
\emph{Israel J.\ Math.}, to appear. arXiv:1509.07914 [math.FA].

\bibitem[GX14]{GaoX:14}
N.~Gao and F.~Xanthos,
Unbounded order convergence and application to martingales without
probability,
\emph{J.\ Math.\ Anal.\ Appl.}, 415 (2014), 931--947.

\bibitem[Gao14]{Gao:14}
N.~Gao, Unbounded order convergence in dual spaces,
\emph{J.\ Math.\ Anal.\ Appl.}, 419, 2014, 347--354.

\bibitem[GLX]{GLX}
N.~Gao, D.H.~Leung, and F.~Xanthos,
The dual representation problem of risk measures.
Preprint. arXiv:1610.08806 [q-fin.MF]

\bibitem[Hal70]{Halmos:70}
P.R.~Halmos,
\emph{Measure Theory}, Springer-Verlag, New York, 1970.

\bibitem[Kap97]{Kaplan:97}
S.~Kaplan,
On Unbounded Order Convergence, 
\emph{Real Anal.\ Exchange} 23(1), 1997, 175--184.

\bibitem[KN63]{Kelley:63}
J.L.~Kelley and I.~Namioka, \emph{Linear topological spaces.}
Van Nostrand Co., Inc., Princeton, N.J. 1963

\bibitem[LT79]{Lindenstrauss:79}
J.~Lindenstrauss and L.~Tzafriri,
\emph{Classical {B}anach spaces. {I}{I}}, Springer-Verlag, Berlin,
  1979.

\bibitem[MN91]{Meyer-Nieberg:91}
P.~Meyer-Nieberg,
\emph{Banach lattices},
Springer-Verlag, Berlin, 1991.

\bibitem[Nak48]{Nakano:48}
H.~Nakano,
Ergodic theorems in semi-ordered linear spaces,
\emph{Ann.\ of Math.\ (2)}, 49, 1948, 538--556.

\bibitem[Roy88]{Royden:88}
 H.L.~Royden,
\emph{Real analysis}, 3rd ed., Macmillan Publishing Company,
   New York, 1988. 

\bibitem[Sch74]{Schaefer:74}
H.H.~Schaefer,
\emph{Banach lattices and positive operators},
Springer-Verlag, Berlin, 1974.

\bibitem[Tro04]{Troitsky:04}
V.G.~Troitsky,
Measures of non-compactness of operators on Banach lattices,
\emph{Positivity}, 8(2), 2004, 165--178.

\bibitem[Wic77]{Wickstead:77}
A.W.~Wickstead,
Weak and unbounded order convergence in Banach lattices,
\emph{J.\ Austral.\ Math.\ Soc.\ Ser.\ A}, 24(3), 1977, 312--319.

\bibitem[Wnuk99]{Wnuk:99}
W.~Wnuk,
\emph{Banach lattices with order continuous norms},
Polish scientific publishers PWN, Warszawa, 1999.

\bibitem[Wnuk13]{Wnuk:13}
W.~Wnuk, 
On the dual positive Schur property in Banach lattices,
\emph{Positivity}, 17(2013), 759--773.

\bibitem[Zab]{Zab}
O.~Zabeti,
Unbounded absolute weak convergence in Banach lattices,
preprint. arXiv:1608.02151 [math.FA].

\end{thebibliography}
\end{document}